\documentclass[EJP]{ejpecp} 

\SHORTTITLE{Argmin process} 

\TITLE{The argmin process of random walks, Brownian motion and L\'evy processes
    } 

\AUTHORS{%
   Jim Pitman \footnote{Department of Statistics, University of California, Berkeley. 
    \EMAIL{pitman@stat.berkeley.edu}}
  \and 
  Wenpin~Tang \footnote{Department of Mathematics, UCLA. 
  \EMAIL{wenpintang@math.ucla.edu} }}


\AMSSUBJ{60G50, 60G51, 60J65} 

\SUBMITTED{August 22, 2017} 
\ACCEPTED{June 6, 2018} 

\VOLUME{23}
\YEAR{2018}
\PAPERNUM{60}
\DOI{10.1214/18-EJP185}

\ABSTRACT{In this paper we investigate the argmin process of Brownian motion $B$ defined by $\alpha_t:=\sup\left\{s \in [0,1]: B_{t+s}=\inf_{u \in [0,1]}B_{t+u} \right\}$ for $t \geq 0$. The argmin process $\alpha$ is stationary, with invariant measure which is arcsine distributed. We prove that $(\alpha_t; t \geq 0)$ is a Markov process with the Feller property, and provide its transition kernel $Q_t(x,\cdot)$ for $t>0$ and $x \in [0,1]$. 
Similar results for the argmin process of random walks and L\'evy processes are derived.
We also consider Brownian extrema of a given length. We prove that these extrema form a delayed renewal process with an explicit path construction. We also give a path decomposition for Brownian motion at these extrema.}

\usepackage{amssymb}
\usepackage{latexsym}
\usepackage{amsbsy}
\usepackage{amsfonts}
\usepackage{hyperref}
\usepackage{enumerate}
\usepackage{appendix}
\usepackage{hhline}
\usepackage{color}
\usepackage{graphicx}

\def\marginpar#1{\ignorespaces}

\DeclareMathOperator\erf{erf}
\DeclareMathOperator\sgn{sgn}
\DeclareMathOperator\Var{Var}

\DeclareMathOperator\argmin{argmin}

\begin{document}


\section{Introduction and main results}
In this paper we are interested in the {\em argmin process} $(\alpha_t; t \geq 0)$ of standard Brownian motion $(B_t; t \geq 0)$. That is,

\begin{equation}
\label{argminBM}
\alpha_t:=\sup\left\{s \in [0,1]: B_{t+s}=\inf_{u \in [0,1]}B_{t+u} \right\} \quad \mbox{for all}~t \geq 0.
\end{equation}
For each $t \geq 0$, $t+\alpha_t$ is the last time at which the minimum of $B$ on $[t,t+1]$ is achieved.
The argmin process $\alpha$ is c\`adl\`ag, and takes values in $[0,1]$. 
Except possible upward jumps, the process $(\alpha_t; t \geq 0)$ drifts down at unit speed. So $(\alpha_t; t \geq 0)$ can be interpreted as a storage process \cite{CP1,CP2}. See also \cite{Cinlar, BRT, EvansPitman} for other examples of storage processes. The argmin process $\alpha$ also appeared as the hydrodynamic limit of a surface growth model \cite{BF}.

\begin{figure}[ht]
\begin{center}
\includegraphics[width=0.70 \textwidth]{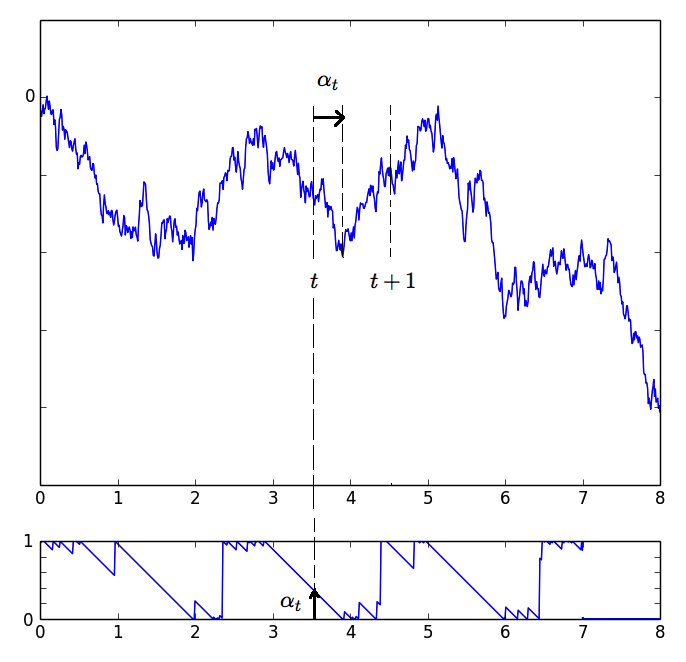}
\end{center}
\caption{TOP: The argmin process embedded in Brownian motion. BOTTOM: The argmin process corresponding to the Brownian path on the top. }
\end{figure}

Let $(\Theta_t; t \geq 0)$ be the $\mathcal{C}[0,\infty)$-valued space-time shift of Brownian motion $B$, defined by
$$\Theta_t: =( B_{t+u} - B_t; u \geq 0) \quad \mbox{for all}~ t \geq 0.$$
In the path space setting, 
\begin{equation}
\label{funcm}
\alpha_t = \alpha_0 \circ \Theta_t \quad \mbox{for all}~t \geq 0.
\end{equation}
By the stationarity of $(\Theta_t; t \geq 0)$, for any measurable function $g: \mathcal{C}[0,\infty) \rightarrow \mathbb{R}$, the process $(g \circ \Theta_t; t \geq 0)$ is stationary. It is a well-known result of L\'{e}vy \cite{Levybook} that the time at which the minimum of Brownian motion on $[0,1]$ is achieved follows the arcsine distribution. As a result, we have the following proposition.
\begin{proposition}
\label{stationary}
The argmin process $(\alpha_t; t \geq 0)$ is stationary. The invariant measure is the arcsine distribution with density 
\begin{equation}
\label{arcsinesta}
f(x)=\frac{1}{\pi \sqrt{x(1-x)}} \quad \mbox{for}~0 < x < 1.
\end{equation}
\end{proposition}

It is natural to ask whether $(\alpha_t; t \geq 0)$ is a Markov process. Sufficient conditions for a function of a Markov process to be Markov are given by Dynkin \cite{Dynkin}, and Rogers and Pitman \cite{RogersPitman}. But these criteria do not apply to the argmin process, see Section \ref{s32}. Nevertheless, we prove the following theorem.
\begin{theorem}
\label{Ftrans}
The argmin process $(\alpha_t; t \geq 0)$ is a Markov process with Feller transition semigroup $Q_t(x, \cdot)$, $t >0$ and $x \in [0,1]$ where
\begin{equation}
\label{Qtrans}
Q_t(x,dy) =  \left\{ \begin{array}{ccl}        
 \displaystyle \frac{1_{\{0<y<1\}}}{ \pi \sqrt{y(1-y)}}dy  & \mbox{for}~0 \leq x \leq 1 <t, \\
 \displaystyle \sqrt{\frac{1-x}{1-x+t}} \delta_{x-t}(dy) + \frac{\sqrt{(y+t-1)^{+}}}{\pi (y+t-x) \sqrt{1-y}} dy & \mbox{for}~0< t \leq x \leq 1\\ 
 \displaystyle \frac{\sqrt{(1-x)(t-x)} + \sqrt{y(y+t-1)^{+}}}{\pi (y+t-x) \sqrt{y(1-y)}}dy  & \mbox{for}~0 \leq x<t \leq 1.
\end{array}\right.
\end{equation}
\end{theorem}
See Kallenberg \cite[Chapter 19]{Kallenberg} for background on Feller semigroups of continuous-time Markov processes. The proof of Theorem \ref{Ftrans} is given in Sections \ref{s33} and \ref{s35}. 

Our approach relies on Denisov's decomposition and excursion theory, see Section \ref{s2}. 
We also investigate the law of jumps of $(\alpha_t; t \geq 0)$. In particular, we prove that the argmin process $\alpha$ has local times at levels $0$ and $1$, and provide a L\'evy system of $(\alpha_t; t \geq 0)$. 
These results imply that the argmin process $\alpha$ is a time-homogeneous Markov process with an explicit description in the framework of {\em jumping Markov processes} by Jacod and Skorokhod \cite{JacodS}, following the study of {\em piecewise deterministic Markov processes} by Davis \cite{Davis}.


The motivation for considering the argmin process comes from the study of Brownian extrema with given length. In the second part of this paper we provide further insight into these extrema, following previous works of Neveu and Pitman \cite{NeveuPitman}, and Leuridan \cite{Leuridan}. For $a,b>0$, let
\begin{equation}
\label{abminima}
\mathcal{M}_{a,b}: = \left\{t \geq a: B_t = \inf_{s \in [t-a,t+b]} B_s\right\} = \{T_1^{a,b}, T^{a,b}_2, \cdots\},
\end{equation}
with $ a < T_1^{a,b}<T^{a,b}_2<\cdots$ be the {\em $(a,b)$-minima set} of Brownian motion $B$. 


The study of Brownian extrema dates back to L\'evy \cite{Levybook}. See \cite[Section 2.9]{KS} for development. Neveu and Pitman \cite{NeveuPitman} proved the renewal property of Brownian local extrema by looking at Brownian extrema of a given depth. They gave the following Palm description of Brownian local extrema. 
\begin{theorem} \cite{NeveuPitman}
\label{NP1}
Let $(\mathcal{C},\mathcal{B})$ be the space of continuous paths on $\mathbb{R}$, equipped with Wiener measure ${\bf W}$, and $(E,\mathcal{E})$ be the space of excursions with lifetime $\zeta$, equipped with It\^{o}'s law ${\bf n}$ (see Section \ref{s22} for discussion).
For $A \in \mathcal{B}$, define the $\sigma$-finite Palm measure of Brownian local minima by
\begin{equation}
\label{nuA}
\nu(A): = \mathbb{E} \#\{0 \leq t \leq 1: t ~\mbox{is a local minimum and}~\theta_t \in A\},
\end{equation}
where $\theta_t: = (b_{t+u} - b_t; t \in \mathbb{R})$ is the space-time shift of a two-sided Brownian motion $b$.
Then for $A \in \mathcal{B}$, 
\begin{equation}
\label{NP2}
\nu(A) = \frac{1}{2} ({\bf n} \otimes {\bf n} \otimes {\bf W})(f^{-1}(A)),
\end{equation}
where $f: E \times E \times \mathcal{C} \ni (e,e',w) \rightarrow  \widetilde{w} \in \mathcal{C}$ is a mapping given by 
\begin{equation}
\label{inversehh}
\widetilde{w}_t =     \left\{ \begin{array}{ccl}
         w_{t+\zeta(e')} & \mbox{if}
         & t \leq -\zeta(e'), \\ e'_{-t}  & \mbox{if} & -\zeta(e') \leq t \leq 0, \\
         e_t & \mbox{if} & 0 \leq t \leq \zeta(e), \\ w_{t-\zeta(e)} & \mbox{if} & t \geq \zeta(e).
                \end{array}\right.
                \end{equation}
\end{theorem}
The quantity $\nu(A)$ is interpreted as the mean number of Brownian local minima of type $A$ per unit time.
See Kallenberg \cite[Chapter 11]{Kallenberg} for background on Palm measures. 

Neveu-Pitman's results were generalized to Brownian motion with drift by Faggionato \cite{Faggionato}. Inspired from \cite{NeveuPitman}, Leuridan \cite{Leuridan} considered Brownian extrema with given length.  
In different directions, Groeneboom \cite{Groeneboom} considered the global extremum of Brownian motion with a parabolic drift, where he gave a density formula in terms of {\em Airy functions}. Tsirelson \cite{Tsirelson} provided an i.i.d. uniform sampling construction of Brownian local extrema under external randomization. Abramson and Evans \cite{AbE} considered Lipschitz minorants of Brownian motion, which is a variant of Brownian extrema.

Leuridan studied the $(a,b)$-minima set $\mathcal{M}'_{a,b}: = \{t \in \mathbb{R}: b_t = \inf_{s \in [t-a, t+b]} b_s\}$ of a two-sided Brownian motion $(b_t; t \in \mathbb{R})$.  He proved that times of the set $\mathcal{M}'_{a,b}$ form a renewal process, and provided the density of the inter-arrival times.  
Observe that
\begin{equation}
\label{stadelay}
\mathcal{M}_{a,b} - a \stackrel{(d)}{=} \mathcal{M}'_{a,b} \cap [0,\infty).
\end{equation}
That is, $\mathcal{M}_{a,b} - a : = \{T_i^{a,b}-a; i \geq 1\}$ is a renewal process with stationary delay. We adapt Leuridan's result to one-sided Brownian motion as follows.

\begin{theorem} \cite{Leuridan}
\label{Leust}
Let $a, b >0$. The times of $(a,b)$-minima of Brownian motion $(B_t; t \geq 0)$ form a delayed renewal process, denoted by $(T_i^{a,b}; i \geq 1)$ so that $a < T_1^{a,b} < T_2^{a,b}< \cdots$. Let 
\begin{equation}
\label{hab}
h_{a,b}(t): =     \left\{ \begin{array}{ccl}
         \displaystyle \frac{1}{\pi t} \left( \sqrt{\frac{(t - b)^{+}}{b}} + \sqrt{\frac{(t - a)^{+}}{a}}\right) & \mbox{for} & 0 < t < a+ b \\ 
         \displaystyle \frac{1}{\pi \sqrt{ab}} & \mbox{for} & t \geq a+b.
                \end{array}\right.
\end{equation}
Then $(\Delta_i^{a,b}: = T_{i+1}^{a,b} - T_i^{a,b}; i \geq 1)$ are independent, with density
\begin{equation}
\label{Gaplaw}
g_{a,b}(t) : = \sum_{n=1}^{\infty} (-1)^{n-1} h_{a,b}^{*n}(t),
\end{equation}
where $h_{a,b}^{*n}$ is the $n^{th}$ convolution of $h_{a,b}$. In addition, $T_1^{a,b}$ is independent of $(\Delta_i^{a,b}; i \geq 1)$, and has density
\begin{equation}
\label{T1law}
f_{a,b}(t) : = \frac{1_{\{t > a\}}}{\pi \sqrt{ab}} \int_{t-a}^{\infty} g_{a,b}(s) ds.
\end{equation}
\end{theorem}

Given a measurable set $A \subset \mathbb{R}^{+}$, let $N_{a,b}(A): = \#(\mathcal{M}_{a,b} \cap A)$ be the counting measure of $(a,b)$-minima in Brownian motion $B$. Leuridan's proof of Theorem \ref{Leust} is based on the formula, for $n \geq 1$ and $0<t_1<\cdots < t_n$,
\begin{equation}
\label{keyleu}
\mathbb{E}(N_{a,b}(dt_1) \cdots N_{a,b}(dt_n)) = \frac{1}{\pi \sqrt{ab}} \prod_{k=1}^{n-1} h_{a,b}(t_k - t_{k-1}) dt_1 \cdots dt_n,
\end{equation}
with convention that $\prod_{\emptyset} := 1$. The case $n=1$ of \eqref{keyleu} follows readily from Theorem \ref{NP1}, since for a generic $(a,b)$-minimum, the left excursion has length larger than $a$ and the right excursion has length larger than $b$. This implies that the mean number of $(a,b)$-minima per unit time is given by
\begin{equation*}
\frac{1}{2} {\bf n}(\zeta(e')>a) {\bf n}(\zeta(e) > b) = \frac{1}{\pi \sqrt{ab}}.
\end{equation*}
In particular,
\begin{equation}
\label{Gapmean}
\mathbb{E}(\Delta_{i}^{a,b} ) = \pi \sqrt{ab} \quad \mbox{for all}~i \geq 1.
\end{equation}
However, to obtain \eqref{keyleu} for $n \geq 2$ requires extra work. 
Observe that for $a+b =1$, the set $\mathcal{M}_{a,b}$ can be viewed as the $a$-level set of the argmin process $\alpha$. So by Brownian scaling,
\begin{align}
\label{Malpha}
\mathcal{M}_{a,b} &\stackrel{(d)}{=} (a+b) \mathcal{M}_{\frac{a}{a+b}, \frac{b}{a+b}} \notag\\
                              & \stackrel{(d)}{=} (a+b) \alpha^{-1}\left( \left\{\frac{a}{a+b}\right\}\right) \quad \mbox{for}~a,b >0.
\end{align}

According to Hoffmann-J{\o}rgensen \cite{HJ}, and Krylov and Ju{\v{s}}kevi{\v{c}} \cite{KJ}, the set $\mathcal{M}_{a,b}$ enjoys regenerative property, and is called a {\em strong Markov set}. See also Kingman \cite{Kingman} for a survey on regenerative phenomena of level sets of Markov processes.
In Section \ref{s42}, we recover Theorem \ref{Leust}, in particular \eqref{keyleu}, by using the properties of the argmin process $\alpha$.

Note that the density $h_{a,b}$ defined by \eqref{hab} is induced by a $\sigma$-finite measure. By Leuridan's formula \eqref{Gaplaw}, the Laplace transform of $\Delta^{a,b}_i$ is given by
\begin{equation}
\label{PbPb}
\Phi_{\Delta^{a,b}}(\lambda) = \frac{\Psi(\lambda)}{1 + \Psi(\lambda)} \quad \mbox{with}~\Psi(\lambda) := \int_{0}^{\infty} e^{-\lambda t} h_{a,b}(t) dt,
\end{equation}
provided that $\Psi(\lambda) < 1$. By analytic continuation, we extend \eqref{PbPb} to all $\lambda>0$.
 But it does not seem obvious how to simplify \eqref{PbPb} analytically. 

While the description of $\mathcal{M}_{a,b}$ is complicated for general $a,b > 0$, the case $a=b$ is simplified. 
For simplicity, we consider $a = b =1$. 
We shall give a construction of the $(1,1)$-minima set
$$\mathcal{M}_{1,1}:= \{T_1, T_2, \cdots\} \quad \mbox{with}~1<T_1<T_2< \cdots,$$
from which we derive simple formulas for the Laplace transforms of $\Delta_i: = T_{i+1} - T_i$ and $T_1$.
 
Let $J$ be the first descending ladder time of Brownian motion, from which starts an excursion above the minimum of length larger than $1$. It is known that the Laplace transform of $J$ is given by 
\begin{equation}
\label{Laplace}
\Phi_J(\lambda) = \frac{1}{\sqrt{\pi \lambda} \erf(\sqrt{\lambda}) + e^{-\lambda}},
\end{equation}
where $\erf(x): = \frac{2}{\sqrt{\pi}} \int_{0}^x e^{-t^2} dt$ is the error function. See Proposition \ref{PropJ} for a derivation of \eqref{Laplace}.

The random variable $J$ plays an important role in our construction of the $(1,1)$-minima set. 
Let $\Delta$ be distributed as the law of the inter-arrival times $\Delta_i$, independent of $J$.
It is a simple consequence of the construction in Section \ref{s43} of the Brownian path over $[0, T_1]$ and over $[T_i,T_{i+1}]$ that
\begin{equation}
\label{001}
T_1 \stackrel{(d)}{=} J + 1_{\{J<1\}} \Delta, \quad \mbox{with}~J~\mbox{independent of}~\Delta.
\end{equation}
Combined with the fact that $T_1 - 1$ is the stationary delay for a renewal process with inter-arrival time distributed according to $\Delta$, this leads to the following result:
\begin{theorem}
\label{minimabetlaw}
Let $(T_i; i \geq 1)$ with $1 < T_1<T_2<\cdots$ be times of the $(1,1)$-minima set $\mathcal{M}_{1,1}$ of Brownian motion $B$. 
\begin{enumerate}
\item
Let $J'$ be an independent copy of $J$, whose Laplace transform is given by \eqref{Laplace}. Then there is the identity in law
\begin{equation}
\label{002}
T_1 - 1 \stackrel{(d)}{=} J+ J'.
\end{equation}
In particular, the Laplace transforms of $T_1$ and $\Delta$ are given by
\begin{equation}
\label{Laplace2}
\Phi_{T_1}(\lambda) =  \frac{e^{-\lambda}}{(\sqrt{\pi \lambda}\erf(\sqrt{\lambda}) + e^{-\lambda})^2},
\end{equation}
and
\begin{equation}
\label{Laplace3}
 \Phi_{\Delta}(\lambda) = 1 -\frac{\pi \lambda}{(\sqrt{\pi \lambda}\erf(\sqrt{\lambda}) + e^{-\lambda})^2}.
\end{equation}
Consequently,
\begin{equation}
\mathbb{E}T_1 = 3 \quad \mbox{and} \quad \mathbb{E} \Delta= \pi.
\end{equation}
\item
The fragments 
$(B_{T_i+t} - B_{T_i}; 0 \leq t \leq \Delta_i)_{i \geq 1}$,
are i.i.d., starting as Brownian meander of length $1$, and then running as Brownian motion until the next $(1,1)$-minima occurs. 

\begin{equation}
\label{mean1level}
\mathbb{E}B_{T_1} = -\sqrt{\frac{\pi}{2}} \quad \mbox{and} \quad \Var B_{T_1} = \frac{4+ \pi}{2},
\end{equation}
and
\begin{equation}
\label{mean2level}
\mathbb{E}(B_{T_{i+1}} - B_{T_i}) = 0 \quad \mbox{and} \quad \Var(B_{T_{i+1}} - B_{T_i}) = \pi~\mbox{for all}~i \geq 1.
\end{equation}
\end{enumerate}
\end{theorem}
In Section \ref{s43}, we prove the identity in law \eqref{002} by computing the Laplace transform \eqref{Laplace2} of $T_1$. This identity in law is surprising, and we do not have a simple explanation. Though we are able to compute the first two moments \eqref{mean1level}-\eqref{mean2level}, the laws of $B_{T_1}$, and $B_{T_{i+1}} - B_{T_i}$ seem to be difficult.
We leave these for further investigation.


Recall the notations in Theorem \ref{NP1}. For $a>0$, let
$$\nu^a(A): = \mathbb{E}\#\{0 \leq t \leq 1: t \in \mathcal{M}'_{a,a}~\mbox{and}~\theta_t \in A\} \quad \mbox{for}~A \in \mathcal{B},$$
be the Palm measure of $(a,a)$-minima of a two-sided Brownian motion. Theorem \ref{minimabetlaw} implies that $\nu^a$ has total mass $\frac{1}{\pi a}$, and
\begin{equation}
\nu^a(A) = \frac{1}{2} 1_{\{\zeta'>a, \zeta>a\}} ({\bf n \otimes n \otimes W})(f^{-1}(A)) \quad \mbox{for}~A \in \mathcal{B},
\end{equation}
where $f$ is the mapping defined by \eqref{inversehh}. By taking $a \downarrow 0$, the Palm measures $\nu^a$ increase to the limit $\nu$ defined by \eqref{NP2}. This recovers Theorem \ref{NP1}.

The set $\mathcal{M}_{1,1}$ is directly related to the argmin process $\alpha$ without scaling. In fact, $T_i$ is the $i^{th}$ time that the process $\alpha$ reaches $0$ by a continuous passage from $1$. So the law of Brownian fragments between $(1,1)$-minima can be derived from the study of $\alpha$. Let 
$$\mathcal{LE}: = \{t \geq 0: B_t < B_s,~ \mbox{for all}~s \in [t,t+1]\},$$
be the left ends of forward meanders of length $1$, and 
$$\mathcal{RE}: = \{t \geq 1: B_t < B_s,~\mbox{for all}~s \in [t-1,t]\},$$
be the right ends of backward meanders of length $1$.
In Lemma \ref{lbefr}, we show that left ends come before right ends between any two consecutive $(1,1)$-minima. So we define for each $i \geq 1$,
\begin{equation}
D_i: = \inf\{t>T_i: t \in \mathcal{RE}\} \quad \mbox{and} \quad G_i:=\sup\{t < D_i: t \in \mathcal{LE}\}.
\end{equation}
For each $i \geq 1$, the triple $(G_i - T_{i}, D_i - G_i, T_{i+1} - D_i)$ gives a decomposition of $\Delta_i$:
\begin{equation}
\label{DecompGDT}
\Delta_i = (G_i - T_{i}) + (D_i - G_i) + (T_{i+1} - D_i).
\end{equation}
By using the L\'evy system of the argmin process, we prove the following theorem which identifies the law of this triple.

\begin{theorem}
\label{ABC}
For each $i \geq 1$, $G_i - T_{i}$, $D_i - G_i$ and $T_{i+1} - D_i$ are mutually independent, with
\begin{itemize}
\item 
$G_i - T_{i} \stackrel{(d)}{=} T_{i+1} - D_i \stackrel{(d)}{=} J$, the Laplace transform of which is given by \eqref{Laplace};
\item
the density of $D_i -G_i$ is given by 
\begin{equation}
\label{lawDG}
\mathbb{P}(D_i -G_i \in dt) =  \frac{2-t}{t^2 \sqrt{t-1}} 1_{\{1<t<2\}}dt.
\end{equation}
Consequently, for each $i \geq 1$,
\begin{equation}
\mathbb{E}(G_i - T_i) = \mathbb{E}(T_{i+1} -D_i) = 1 \quad \mbox{and} \quad \mathbb{E}(D_i-G_i) = \pi -2.
\end{equation}
\end{itemize}
\end{theorem}

\begin{figure}[ht]
\begin{center}
\includegraphics[width=0.75 \textwidth]{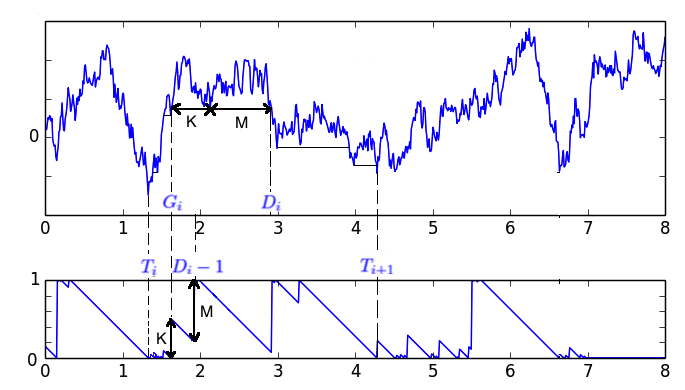}
\end{center}
\caption{$(T_i, G_i, D_i, T_{i+1})$ related to the argmin process $(\alpha_t; t \geq 0)$.}
\end{figure}

For a random variable $X$, let $\Phi_X(\lambda)$ be the Laplace transform of $X$, and $f_X(\cdot)$ be the density of $X$. Theorems \ref{Leust}-\ref{ABC} provide three different descriptions of the inter-arrival time $\Delta$. This leads to some non-trivial identities.
We summarize the results in the following table. 
    \vskip 10 pt
\begin{center}
    \begin{tabular}{| c | c | c |}
    \hline
               & & \\
               & Laplace transform $\Phi_{X}(\lambda)$ & Density $f_{X}(t)$  \\ 
               & &  \\ \hhline{|=|=|=|}
               & & \\
     $X = J$, Prop \ref{PropJ}  & $\displaystyle \frac{1}{\sqrt{\pi \lambda} \erf(\sqrt{\lambda}) + e^{-\lambda}}$ & Given by \eqref{PYdensity} \\ 
     & & \\ \hhline{|=|=|=|}
     & & \\
     $X = D - G$, Th \ref{ABC} & $\displaystyle \pi \lambda  \left[(\erf(\sqrt{\lambda})^2 -1)\right]+ 2 \sqrt{\pi \lambda} e^{-\lambda} \erf(\sqrt{\lambda})+ e^{-2\lambda}$ & $\displaystyle \frac{2-t}{t^2 \sqrt{t-1}} 1_{\{1<t<2\}}$ \\ 
     & & \\ \hhline{|=|=|=|}
     & & \\ 
    $X = \Delta$, Th \ref{Leust} & $\displaystyle \frac{\Psi(\lambda)}{1+ \Psi(\lambda)}$, with  & Given by \eqref{Gaplaw}, \\ 
    & $\displaystyle \Psi(\lambda) = \erf(\sqrt{\lambda})^2 - 1 + \frac{2 e^{-\lambda}}{\sqrt{\pi \lambda}} \erf(\sqrt{\lambda}) + \frac{e^{-2 \lambda}}{\pi \lambda}$ & with $a=b=1$  \\ 
    & & \\ \hline
    & & \\
     $X = \Delta$, Th \ref{minimabetlaw} &  $\displaystyle 1 - \pi \lambda (\Phi_J(\lambda))^2$ & $\displaystyle -\pi \frac{d}{dt} f^{*2}_J(t)$ \\
     & &   \\\hline
     & & \\
     $X = \Delta$, Th \ref{ABC} &  $\displaystyle (\Phi_J(\lambda))^2 \Phi_{D-G}(\lambda)$  &  $(f_J*f_{D-G}*f_J)(t)$   \\
     & & \\ \hhline{|=|=|=|}
     & & \\
     $X = T_1 -1$, Th \ref{Leust} & $\displaystyle \frac{1}{\pi \lambda} (1 - \Phi_{\Delta}(\lambda))$ &  $ \displaystyle \frac{1_{\{t>0\}}}{\pi} \int_{t}^{\infty} f_{\Delta}(s) ds$  \\ 
     & & \\ \hline
     & & \\
     $X = T_1-1 $, Th \ref{minimabetlaw} & $\displaystyle (\Phi_J(\lambda))^2$ & $f_J^{*2}(t)$ \\ 
     & & \\ \hline
    \end{tabular}\\
    \vskip 8 pt
    {TABLE $1$. Laplace transforms and densities.}
\end{center}

Finally, we extend Theorem \ref{Ftrans} to random walks and L\'evy processes. Fix $N \geq 1$. We study the argmin process $(A_N(n); n \geq 0)$ of a random walk $(S_n; n \geq 0)$, defined by
\begin{equation}
\label{argminchain}
A_N(n):=\sup \left\{1 \leq i \leq N; S_{n+i} = \min_{1\leq i \leq N} S_{n+i} \right\} \quad \mbox{for all}~n \geq 0,
\end{equation}
where $S_n:=\sum_{i=1}^n X_i$ is the $n^{th}$ partial sum of $(X_n; n \geq 1)$ (with convention $S_0:=0$), and $(X_n; n \geq 1)$ is a sequence of independent and identically distributed random variables with the cumulative distribution function $F$. This is  the discrete analog of the argmin process of Brownian motion.
A similar argument as in the Brownian case shows that $(A_N(n); n \geq 0)$ is a Markov chain. For $n \geq 1$, let
$$p_n: = \mathbb{P}(S_1 \geq 0, \cdots, S_n \geq 0) \quad \mbox{and} \quad \widetilde{p}_n: = \mathbb{P}(S_1 > 0, \cdots, S_n > 0).$$
Theorem \ref{GF} below recalls the classical theory of how the two sequences of probabilities $p_n$ and $\widetilde{p}_n$ are determined by the sequences of probabilities $\mathbb{P}(S_n \ge 0 )$ and $\mathbb{P}( S_n > 0 )$.
We give the transition matrix of the argmin chain $A_N$ in terms of $(p_n; n \geq 1)$ and $(\widetilde{p}_n; n \geq 1$), which can be made explicit for special choices of $F$. 

\begin{theorem}
\label{discreteth}
Whatever the common distribution $F$ of $(X_n; n \geq 0)$, the argmin chain $(A_N(n); n \geq 0)$ is a stationary and time-homogeneous Markov chain on $\{0,1, \ldots, N\}$. Let $\Pi_N(k)$, $k \in [0,N]$ be the stationary distribution, and $P_N(i,j)$, $i,j \in [0,N]$ be the transition probabilities of the argmin chain $(A_N(n); n \geq 0)$ on $[0,N]$. Then
\begin{equation}
\label{stationaryPi}
\Pi_N(k) = p_k \widetilde{p}_{N-k} \quad \mbox{for}~0 \leq k \leq N;
\end{equation}
\begin{equation}
\label{trans1}
P_N(i,N ) = 1 -\frac{\widetilde{p}_{N+1-i}}{\widetilde{p}_{N-i}} \quad \mbox{and} \quad P_N(i,i-1) = \frac{\widetilde{p}_{N+1-i}}{\widetilde{p}_{N-i}} \quad \mbox{for}~0<i \leq N;
\end{equation}
\begin{equation}
\label{trans2}
P_N(0,j ) = \frac{(p_j-p_{j+1}) \widetilde{p}_{N-j}}{\widetilde{p}_{N}} \quad \mbox{for}~0\leq  j < N \quad \mbox{and} \quad P_N(0,N) = 1 - \sum_{j=0}^{N-1} P_N(0,j).
\end{equation}
Consequently,
\begin{enumerate}
\item 
If $(S_n; n \geq 0)$ is a random walk with continuous distribution and $\mathbb{P}(S_n>0) = \theta \in (0,1)$ for all $n \geq 1$. Let $(\theta)_{n \uparrow}: = \prod_{i=0}^{n-1} (\theta + i)$ be the Pochhammer symbol. Then
\begin{equation}
\label{disarcsine}
\Pi_{N}(k) = \frac{(\theta)_{k \uparrow} (\theta)_{N-k \uparrow}}{k! (N-k)!} \quad \mbox{for}~0 \leq k \leq N;
\end{equation}
\begin{equation}
\label{eq5354}
P_N(i, N) = \frac{1-\theta}{N+1-i} \quad \mbox{and} \quad P_N(i,i-1) = \frac{N+\theta-i}{N+1-i} \quad \mbox{for}~0<i \leq N;
\end{equation}
\begin{equation}
\label{eq55}
P_N(0, j) =  \frac{1 - \theta}{j+1} \binom{N}{j} \frac{(\theta)_{j \uparrow} (\theta)_{N-j \uparrow}}{(\theta)_{N \uparrow}} \quad \mbox{for}~0 \leq j < N;
\end{equation}
and
\begin{equation}
\label{eq56}
P_N(0,N) = \frac{2(1-\theta)}{N+1} - \frac{(1 - 2 \theta)(2 \theta)_{N \uparrow}}{(N+1)(\theta)_{N \uparrow}}.
\end{equation}
\item
If $(S_n; n \geq 0)$ is a simple symmetric random walk. Let $\lfloor x \rfloor$ be the integer part of $x$. Then
\begin{equation}
\label{eq513}
\Pi_N(k) =   \frac{\displaystyle \left(\frac{1}{2}\right)_{ \lfloor{\frac{k+1}{2}}\rfloor \uparrow} \left(\frac{1}{2}\right)_{\lfloor \frac{N-k}{2} \rfloor \uparrow} }{\displaystyle 2 \cdot \left \lfloor \frac{k+1}{2} \right \rfloor ! \left \lfloor \frac{N-k}{2}  \right \rfloor !}  \quad \mbox{for}~0 \leq k \leq N;
\end{equation}
For $0 < i \leq N$,
\begin{equation}
\label{eq514}
P_N(i,N) =     \left\{ \begin{array}{ccl}
         \frac{N-i}{N+1-i} & \mbox{if}~ N-i ~\mbox{is odd}; \\ [8 pt]
         1 & \mbox{if} ~N-i ~\mbox{is even};
                \end{array}\right.
\mbox{and} ~
P_N(i, i-1) = 1 - P_N(i,N);
\end{equation}
for $0 \leq j <N$,
\begin{equation}
\label{eq515}
P_N(0,j) = \left\{ \begin{array}{ccl}
          0 & \mbox{if}~ j ~\mbox{is odd}; \\ [8 pt]
         \frac{\displaystyle \binom{j}{\frac{j}{2}} \binom{2 \lfloor \frac{N}{2} \rfloor-j}{\lfloor \frac{N}{2} \rfloor-\frac{j}{2}}}{\displaystyle(j+2)\binom{2 \lfloor \frac{N}{2} \rfloor}{ \lfloor\frac{N}{2} \rfloor}} & \mbox{if} ~j ~\mbox{is even};
                \end{array}\right.
\end{equation}
and
\begin{equation}
\label{eq516}
P_N(0,N) = \left\{ \begin{array}{ccl}
         \frac{1}{N+1} & \mbox{if}~ N ~\mbox{is odd}; \\ [8 pt]
         \frac{2}{N+2} & \mbox{if} ~N ~\mbox{is even}.
                \end{array}\right.
\end{equation}
\end{enumerate}
\end{theorem}
For the argmin chain $A_N$, the transition probability from $0$ to $N$ is given by \eqref{trans2} in the general case. But this probability is simplified to \eqref{eq56} and \eqref{eq516} in the two special cases. These identities are proved analytically by Lemmas \ref{lem11} and \ref{lem12}. We do not have a simple explanation, and leave combinatorial interpretations for the readers.

Let $(X_t; t \geq 0)$ be a real-valued L\'evy process. We consider the argmin process $(\alpha_t^X; t \geq 0)$ of $X$, defined by
\begin{equation}
\alpha_t^X : = \sup\left\{s \in [0,1]: X_{t+s}=\inf_{u \in [0,1]}X_{t+u} \right\} \quad \mbox{for all}~t \geq 0.
\end{equation}
We are particularly interested in the case where $X$ is a stable L\'evy process. We follow the notations in Bertoin \cite[Chapter VIII]{Bertoin}. Up to a multiple factor, a stable L\'evy process $X$ is entirely determined by a {\em scaling parameter} $\alpha \in (0,2]$, and a {\em skewness parameter} $\beta \in [-1,1]$. The characteristic exponent of a stable L\'evy process $X$ with parameters $(\alpha,\beta)$ is given by
\begin{equation*}
    \Psi(\lambda) : = \left\{ \begin{array}{ccl}        
  |\lambda|^{\alpha} (1 - i \beta \sgn(\lambda) \tan(\pi \alpha/2))  & \mbox{for}~\alpha \neq 1, \\ [8pt]
  |\lambda| (1 + i \frac{2 \beta}{\pi} \sgn(\lambda) \log|\lambda|)  & \mbox{for}~\alpha = 1
\end{array}\right.
\end{equation*}
where $\sgn$ is the sign function. Let $\rho : = \mathbb{P}(X_1>0)$ be the {\em positivity parameter}. Zolotarev \cite[Section 2.6]{Zolo} found that 
\begin{equation}
\label{positivity}
\rho = \frac{1}{2} + (\pi \alpha)^{-1}\arctan(\beta \tan(\pi \alpha/2)) \quad \mbox{for}~\alpha \in (0,2].
\end{equation}
If $X$ (resp. $-X$ ) is a subordinator, then almost surely $\alpha^X_t =0$ (resp. $\alpha^X_t = 1$) for all $t \geq 0$. The following theorem is a generalization of Theorem \ref{Ftrans}.

\begin{theorem}
\label{mainLevy}
~
\begin{enumerate}
\item
Let $(X_t; t \geq 0)$ be a L\'evy process. Then the argmin process $(\alpha^X_t; t \geq 0)$ of $X$ is a stationary and time-homogeneous Markov process.
\item
Let $(X_t; t \geq 0)$ be a stable L\'evy process with parameters $(\alpha,\beta)$, and assume that neither $X$ nor $-X$ is a subordinator. 
Let $\rho$ be defined by \eqref{positivity}. Then the argmin process $(\alpha^X_t; t \geq 0)$ of $X$ has generalized arcsine distributed invariant measure whose density is
\begin{equation}
f(x) : = \frac{\sin \pi \rho}{\pi} x^{-\rho} (1-x)^{\rho-1} \quad \mbox{for}~0<x<1,
\end{equation}
and Feller transition semigroup $Q^X_t(x, \cdot)$, $t >0$ and $x \in [0,1]$ where
\begin{equation}
\label{Qtrans}
Q^X_t(x,dy) =  \left\{ \begin{array}{ccl}        
 1_{\{0<y<1\}} \frac{\sin \pi \rho}{\pi} y^{-\rho} (1-y)^{\rho-1} dy  & \mbox{for}~0 \leq x \leq 1 <t, \\ [8pt]
 \left(\frac{1-x}{1-x+t}\right)^{1-\rho} \delta_{x-t}(dy) +  \frac{\sin \pi \rho}{\pi} \cdot \frac{ (1-y)^{\rho-1} (y+t-1)_{+}^{\rho} }{(y+t-x)}dy & \mbox{for}~0< t \leq x \leq 1,\\ [8pt]
 \frac{\sin \pi \rho}{\pi (y+t-x)} y^{-\rho} (1-y)^{\rho-1} [(t-x)^{\rho}(1-x)^{1-\rho} + y^{\rho}(y+t-1)_{+}^{1-\rho}] dy  & \mbox{for}~0 \leq x<t \leq 1.
\end{array}\right.
\end{equation} 
\end{enumerate}
\end{theorem}

\vskip 12 pt
{\bf Organization of the paper:} The layout of the paper is as follows. 
\begin{itemize}
\item
In Section \ref{s2}, we provide background and necessary tools which will be used later. 
\item
In Section \ref{s3}, we study the argmin process $(\alpha_t; t \geq 0)$ of Brownian motion, and prove Theorem \ref{Ftrans}.
\item
In Section \ref{s4}, we study the $(a,b)$-minima of Brownian motion with an emphasis on the case $a=b=1$. There we prove Theorems \ref{Leust}, \ref{minimabetlaw} and \ref{ABC}.
\item 
In Section \ref{s5}, we consider the argmin process of random walks and L\'evy processes, and prove Theorems \ref{discreteth} and \ref{mainLevy}.
\end{itemize}
\section{Background and tools}
\label{s2}
This section recalls some background of Brownian motion. In Section \ref{s21}, we consider Denisov's decomposition for Brownian motion. In Section \ref{s22}, we recall various results from Brownian excursion theory. 
\subsection{Path decomposition of Brownian motion}
\label{s21}
Let $(B_t; t \geq 0 )$ be standard Brownian motion. A Brownian meander $(m_t; 0 \leq t \leq 1)$ can be regarded as the weak limit of 
$$\left(B_t; 0 \leq t \leq 1 \Bigg|\inf_{0 \leq s \leq 1}B_s>-\epsilon \right) \quad \mbox{as} ~\epsilon \downarrow 0.$$
We refer to Durrett et al. \cite{DIM} for a proof. A Brownian meander of length $x$, say $(m^x_t; 0 \leq t \leq x)$ is defined as
\begin{equation*}
m^x_t:= \sqrt{x} m_{t / x} \quad \mbox{for}~0 \leq t \leq x.
\end{equation*}
In particular, $m^x_x$ is Rayleigh distributed with scale parameter $\sqrt{x}$. That is,
\begin{equation}
\label{meanderRay}
\mathbb{P}(m^x_x \in dy) = \frac{y}{x} \exp\left(-\frac{y^2}{2x}\right)dy \quad \mbox{for}~y>0.
\end{equation}
Consequently, 
\begin{equation}
\label{RayEV}
\mathbb{E}m^x_x = \sqrt{\frac{\pi x}{2}} \quad \mbox{and} \quad \Var m^x_x = \frac{4-\pi}{2} x.
\end{equation}

The following path decomposition is due to Denisov.
\begin{theorem}[Denisov's decomposition for Brownian motion,  \cite{Denisov}]
\label{DenisovBM} 
Let $A: = \argmin_{u \in [0,1]} B_u$ be the time at which Brownian motion $B$ attains its minimum on $[0,1]$. Given $A$, which is arcsine distributed,
the Brownian path is decomposed into two conditionally independent pieces:
\begin{enumerate}[(a).]
\item
$(B_{A-t}-B_{A}; 0 \leq t \leq A)$ is a Brownian meander of length $A$;
\item
$(B_{A+t}-B_{A}; 0 \leq t \leq 1-A)$ is a Brownian meander of length $1-A$.
\end{enumerate}
\end{theorem}

\quad Let\begin{equation}
\label{defPx}
{\bf P}^x : = \overset{\longleftarrow}{{\bf M}^x} \otimes \overset{\longrightarrow}{{\bf M}^{1-x}} \otimes {\bf W} \quad \mbox{for}~0 \leq x \leq 1
\end{equation}
be the law of two independent Brownian meanders of length $x$ and $1-x$ joined back-to-back, concatenated by an independent Brownian path running forever. Denisov's decomposition is equivalent to
\begin{equation}
{\bf W}(\cdot) = \int_0^1 \frac{1}{\pi \sqrt{x(1-x)}} {\bf P}^x(\cdot) dx.
\end{equation}
\subsection{Brownian excursion theory}
\label{s22}
Let $(B_t; t \geq 0)$ be standard Brownian motion, and $\underline{B}_t:= \inf_{0 \leq s \leq t} B_s$ be the past-minimum process of $B$. For $l>0$, let $T_l:= \inf\{t>0; B_t<-l\}$ be the first time at which $B$ hits below level $-l$. Let 
$$\mathcal{D}:=\{l>0: T_{l-}<T_l\},$$
so that for $l \in \mathcal{D}$, 
$$e_l: = \left\{ \begin{array}{cll}        
 B_{T_{l-}+t} - \underline{B}_{T_{l-}+t}  & \mbox{for}~0 \leq t \leq T_l - T_{l-} \\ 
 0 & \mbox{for}~t > T_l - T_{l-}
\end{array}\right.
$$
is an excursion away from $-l$. Let $E$ be the space of excursions defined by
$$E:=\{\epsilon \in \mathcal{C}[0,\infty); \epsilon_0=0,~\epsilon_t>0~\mbox{for}~t \in (0,\zeta(\epsilon)),~\mbox{and}~\epsilon_t=0~\mbox{for}~t \geq \zeta(\epsilon)\},$$
where $\zeta(\epsilon):=\inf\{t>0; \epsilon_t=0\} \in (0,\infty)$ is the lifetime of the excursion $\epsilon \in E$. The following theorem is a special case of It\^{o}'s excursion theory.

\begin{theorem} \cite{Itoex}
\label{Itoexth}
The point measure
$$\sum_{l \in \mathcal{D}} \delta_{(l,e_l)}(dsd\epsilon)$$
is a Poisson point process on $\mathbb{R}^{+} \times E$ with intensity $ds \times {\bf n}(d \epsilon)$, where ${\bf n}(d \epsilon)$, called It\^{o}'s excursion law, is a $\sigma$-finite measure on $E$.
\end{theorem}
Here we consider positive excursions of the reflected process $B - \underline{B}$. So the measure ${\bf n}(d\epsilon)$ corresponds to $2 {\bf n}^{+}(d\epsilon)$ defined in Revuz and Yor \cite[Chapter XII]{RY}.

Let $\Lambda(dx)$ be the L\'{e}vy measure of a $\frac{1}{2}-$stable subordinator such that
\begin{equation}
\label{halfss}
\Lambda(dx) = \frac{dx}{\sqrt{2 \pi x^3}} \quad \mbox{and} \quad \Lambda(x,\infty) = \sqrt{\frac{2}{\pi x}}\quad \mbox{for}~x>0.
\end{equation}
By applying the master formula of Poisson point processes, we know that
\begin{equation}
\label{Itolength}
{\bf n} (\zeta \in dx) = \Lambda(dx).
\end{equation}
See Revuz and Yor \cite[Chapter XII]{RY} for development of Brownian excursion theory. 
Let 
\begin{equation}
\label{age}
A_t: = \sum_{l} 1_{\{T_{l-} \leq t \leq T_l\}} (t - T_{l-}),
\end{equation}
be the {\em age process} of excursions of $B-\underline{B}$, or equivalently of a $\frac{1}{2}$-stable subordinator.
The following proposition gathers useful results of $J$, defined by
$$J + 1:= \inf\{t > 0: A_t = 1\}.$$
That is, $J$ is the first descending ladder time of Brownian motion, from which starts an excursion above the minimum of length exceeding $1$. 
For completeness, we include a proof.
\begin{proposition}
\cite{PYbis}
\label{PropJ}
Let $J$ be the first descending ladder time of Brownian motion, from which starts an excursion above the minimum of length exceeding $1$. 
\begin{enumerate}
\item
The random variable $\frac{1}{1+J}$ has the same distribution as the longest interval of Poisson-Dirichlet $ (\frac{1}{2},0)$ distribution.
The Laplace transform of $J$ is given by \eqref{Laplace}, and
\begin{equation}
\label{meanfirst}
\mathbb{E}J = 1.
\end{equation}
\item
The density of $J$ is given by 
\begin{equation}
\label{PYdensity}
\frac{\mathbb{P}(J \in dt)}{dt} = \sum_{n=1}^{\infty} (-1)^{n+1}\frac{c_n}{\sqrt{t}} I_n\left(\frac{1}{1+t}\right) \quad \mbox{for}~t>0,
\end{equation}
where for each $n \geq 1$,
\begin{equation*}
c_n : = \frac{(n-1)!}{2^{n-1} \pi ^{n/2} \Gamma(n/2)},
\end{equation*}
and $I_n$ is a function supported on $(0,\frac{1}{n}]$ defined by
\begin{equation*}
I_n(u_{n}): =  \int \prod_{i=1}^{n-1} \sqrt{\frac{(1-u_i)^{n-1-i}}{u_i^{3}}} 1_{\left\{\frac{u_{i+1}}{1-u_{i+1}} \leq u_i \leq \frac{1}{i}\right\}} du_i \quad \mbox{for}~u_n \in \left(0, \frac{1}{n}\right],
\end{equation*}
with convention that $I_1(u_1) : = 1$ for $u_1 \in (0,1]$. 
Consequently,
\begin{equation}
\label{firstlaw}
\mathbb{P}(J \in dt) = \frac{dt}{\pi \sqrt{t}} \quad \mbox{for}~0<t \leq 1.
\end{equation}
\end{enumerate}
\end{proposition}
\begin{proof}
The part $(1)$ is essentially from Pitman and Yor \cite[Corollary 12]{PYbis} with $\alpha = \frac{1}{2}$. Alternatively, let
$\tau := \inf\{l \in D: T_l - T_{l-}>1\}$ be the first level above which an excursion has length larger than $1$ so that $J = T_{\tau-}$.
As in \cite{GreenPit}, we deduce from Theorem \ref{Itoexth} that $\tau$ is exponentially distributed with rate $\Lambda(1,\infty)=\sqrt{2/\pi}$, independent of  $(B-\underline{B})[I_{\tau}]: = (B_t - \underline{B}_t; t \in I_{\tau})$ and that 
\begin{equation*}
J \stackrel{(d)}{=} \sigma_{\xi},
\end{equation*}
where 
\begin{itemize}
\item
$(\sigma_t; t \geq 0)$ is a $\frac{1}{2}-$stable subordinator with all jumps of size larger than $1$ deleted, so the Laplace exponent of $(\sigma_t; t \geq 0)$ is given by
$$\phi(\lambda) := \int_0^1 (1-e^{-\lambda x}) \Lambda(dx) = \sqrt{2 \lambda} \erf(\sqrt{\lambda}) - \sqrt{\frac{2}{\pi}}(1-e^{-\lambda}) \quad \mbox{for}~\lambda \geq 0.$$
\item
$\xi$ is exponentially distributed with rate $\sqrt{2/\pi}$, independent of $(\sigma_t; t \geq 0)$. 
\end{itemize}
So 
\begin{equation*}
\mathbb{E}J = \mathbb{E}\sigma_1 \mathbb{E}\xi = \sqrt{\frac{2}{\pi}} \cdot \sqrt{\frac{\pi}{2}} = 1,
\end{equation*}
and the Laplace transform of $J$ is given by
\begin{equation*}
\Phi_{J}(\lambda) = \frac{\sqrt{2/\pi}}{\sqrt{2/\pi} + \phi(\lambda)} = \frac{1}{\sqrt{\pi \lambda} \erf(\sqrt{\lambda}) + e^{-\lambda}}.
\end{equation*}
The part $(2)$ is obtained by specializing Pitman and Yor \cite[Proposition 20]{PYbis} to $\alpha = \frac{1}{2}$ and $\theta = 0$. 
\end{proof}

The following result can be read from Maisonneuve \cite[Section 8]{Maison} and Bolthausen \cite{Bolmeander}. An alternative approach was provided by Greenwood and Pitman \cite{GreenPit}, and Pitman \cite[Sections 4 and 5]{PitmanLevy}. 
\begin{theorem} \cite{Maison, Bolmeander}
\label{exmeander}
The process 
$$(B-\underline{B})[J, J+1]: = (B_{J+t} - \underline{B}_{J+t}; 0 \leq t \leq 1)$$
is a Brownian meander of length $1$, independent of $(B_u; 0 \leq u \leq J)$.
\end{theorem}
\section{The argmin process of Brownian motion}
\label{s3}
In this section, we study the argmin process $\alpha$ of Brownian motion defined by \eqref{argminBM}. In Section \ref{s31}, we deal with the sample path properties of $\alpha$. In Section \ref{s33}, we provide a conceptual proof that the argmin process $\alpha$ is a Markov process with the Feller property. In Section \ref{s34}, we study the jumps of $\alpha$ by means of a L\'evy system. In Section \ref{s35}, we compute the transition kernel of $\alpha$, and prove Theorem \ref{Ftrans}. Finally in Section \ref{s32}, we explain why Dynkin's criterion and the Rogers-Pitman criterion do not apply to the argmin process $\alpha$.
\subsection{Sample path properties}
\label{s31}
We have mentioned in the introduction that the argmin process $(\alpha_t; t \geq 0)$ takes values in $[0,1]$, and drifts down at unit speed except for positive jumps. More precisely, we provide the following proposition.
\begin{proposition} 
\label{argminBMsp}
Let $(\alpha_t; t \geq 0)$ be the argmin process of Brownian motion. Then a.s.
 \begin{enumerate}
 \item \label{boundedincb} 
 $\alpha_t \in [0,1]$ for all $t \geq 0$, and $(t+\alpha_t; t \geq 0)$ is increasing;
 \item \label{jumptypb}
$(\alpha_t; t \geq 0)$ decreases at unit speed except for

$(i)$ jumps from $0$ to some $x \in (0,1)$;

$(ii)$ jumps from some $x \in (0,1)$ to $1$.
 \end{enumerate}
 \end{proposition}
\begin{proof}
\eqref{boundedincb} The fact $\alpha_t \in [0,1]$ is straightforward from the definition. Let $0 \leq t < t'$. 
\begin{itemize}
\item
If $t'> t+ \alpha_t$, then $t'+\alpha_{t'} >  t+ \alpha_t$. 
\item
If $t' \leq t+ \alpha_t$, then $B_{t+\alpha_t} \leq  B_u$ for all $u \in [t', t+\alpha_t]$. This implies that $\alpha_{t'} \geq t+\alpha_t-t'$. 
\end{itemize}
\eqref{jumptypb} Observe that $(\alpha_t; t \geq 0)$ is a c\`{a}dl\`{a}g process with only positive jumps. We first check $(i)$. If $\alpha_{t-}=0$ for some $t>0$, then $B_u \geq B_t$ for all $u \in [t,t+1]$. We distinguish two cases. In the first case, $B_u > B_t$ for all $u \in [t,t+1]$, which implies that $\alpha_t = 0$. In the second case, $B_u = B_t$ for some $u \in (t,t+1]$, and let $x: = \sup\{u \in (0,1]: B_{t+u} = B_t\} $. If $x=1$, then there exists an excursion of length $1$ in Brownian motion by a space-time shift. But this is excluded by Pitman and Tang \cite[Theorem 4]{PTpattern}. Thus, $\alpha_t = x \in (0,1)$.
\begin{figure}[ht]
\includegraphics[width=1\textwidth]{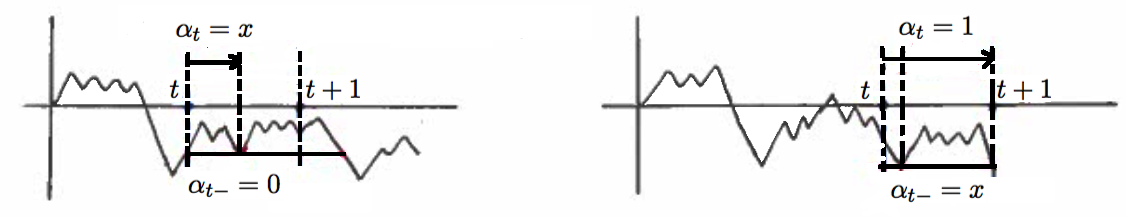}
\caption{LEFT: A jump from $0$ to some $x \in (0,1)$ in the argmin process. RIGHT:  A jump from some $x \in (0,1)$ to $1$ in the argmin process.}
\end{figure}

It remains to check $(ii)$. If $\alpha_{t-}=x \in (0,1)$ for some $t > 0$, then $B_{t+x} < B_u$ for all $u \in (t+x, t+1)$. We also distinguish two cases. In the first case, $B_{t+1} > B_{t+x}$ which implies that $\alpha_t = x$. In the second case, $B_{t+1} = B_{t+x}$ which yields $\alpha_t = 1$.
\end{proof}

Next we prove a time reversal property of the argmin process $\alpha$. By convention, $\alpha_{0-} = \alpha_0$.
\begin{proposition}
\label{timereversal}
For each fixed $T>0$, $(1-\alpha_{(T-t) -}; 0 \leq t \leq T)$ has the same distribution as $(\alpha_t; 0 \leq t \leq T)$.
\end{proposition}
\begin{proof}
Observe that $(1 -\alpha_{(T-t) -}; 0 \leq t \leq T)$ is also a c\`adl\`ag process. Let 
\begin{equation*}
\widetilde{B} :=(B_{T+1-u}-B_{T+1}; 0 \leq u \leq T+1) \stackrel{(d)}{=} (B_u; 0 \leq u \leq T+1).
 \end{equation*}
Let $\widetilde{\alpha}$ be the argmin process of $\widetilde{B}$ on $[0,T]$. Hence,
$$(1-\alpha_{(T-t) -}; 0 \leq t \leq T) = (\widetilde{\alpha}_t; 0 \leq t \leq T) \stackrel{(d)}{=} (\alpha_t; 0 \leq t \leq T).$$
\end{proof} 

\subsection{Markov and Feller property}
\label{s33}
We provide a soft argument to prove that $(\alpha_t; t \geq 0)$ is a Markov process, and enjoys the Feller property. 

For each $t \geq 0$, let
\begin{equation}
\mathcal{G}_t: = \sigma(B_{s}; 0 \leq s \leq t + \alpha_t),
\end{equation}
be the $\sigma$-field generated by the path $B$ killed at time $t+ \alpha_t$. 
By Proposition \ref{argminBMsp} \eqref{boundedincb}, $t \mapsto t+\alpha_t$ is increasing.
It is not hard to see that for any $s<t$, $s + \alpha_s$ is a measurable function of the path $B$ killed at $t + \alpha_t$. So $(\mathcal{G}_t)_{t \geq 0}$ is a filtration. Now we show that
\begin{proposition}
\label{Markov}
The argmin process $(\alpha_t; t \geq 0)$ is time-homogeneous Markov with respect to $(\mathcal{G}_t)_{t \geq 0}$. 
\end{proposition}
\begin{proof}
Fix $t>0$. By Denisov's decomposition (Theorem \ref{DenisovBM}), given $\alpha_t=x$ (i.e. the minimum of $B$ on $[t,t+1]$ is attained at $t+x$), the Brownian path is decomposed into four independent components:
\begin{itemize}
\item
$(B_{t-s}-B_t; 0 \leq s \leq t)$ is Brownian motion of length $t$;
\item
$(B_{t+x-s}-B_{t+x}; 0 \leq s \leq x)$ is a Brownian meander of length $x$;
\item
$(B_{t+x+s}-B_{t+x}; 0 \leq s \leq 1-x)$ is a Brownian meander of length $1-x$;
\item
$(B_{t+1+s}-B_{t+1}; s \geq 0)$ is Brownian motion running forever.
\end{itemize}
\begin{figure}[ht]
\begin{center}
\includegraphics[width=0.65 \textwidth]{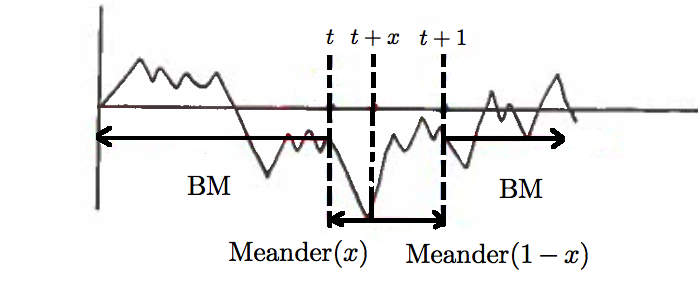}
\end{center}
\caption{Decomposition of Brownian motion given $\alpha_t=x$.}
\end{figure}
As a consequence, $(B_{t+x+s}-B_{t+x}; s \geq 0)$ and $(B_{t+x-s}-B_{t+x}; 0 \leq s \leq t+x)$ are conditionally independent. Observe that given $\alpha_t=x$,
\begin{itemize}
\item
for $s>t$, the minimum of $B$ on $[s,s+1]$ cannot be attained on $[s,t+x)$. So $(\alpha_s; s>t)$ is entirely determined by the path $(B_{t+x+s}-B_{t+x}; s \geq 0)$.
\item
for $s<t$, the minimum of $B$ on $[s,s+1]$ cannot be attained on $(t+x,s+1]$. So $(\alpha_s; s<t)$ is entirely determined by the path $(B_{t+x-s}-B_{t+x}; 0 \leq s \leq t+x)$.
\end{itemize}
These observations imply that $(\alpha_t; t \geq 0)$ is Markov relative to $(\mathcal{G}_t)_{t \geq 0}$. The time-homogeneity follows from the fact that given $\alpha_t=x$, the law of $(B_{t+x+s}-B_{t+x}; s \geq 0)$ does not involve the time parameter $t$.
\end{proof}

We now investigate the Feller property of the argmin process $(\alpha_t; t \geq 0)$. Recall the definition of ${\bf P}^x$ from \eqref{defPx}. Let
\begin{equation}
\label{argmincd}
\alpha_t^x : =\alpha_t(B) \quad \mbox{for}~t \geq 0,
\end{equation}
 be the argmin process of $(B_t; t \geq 0)$ under ${\bf P}^x$, which makes $\alpha_0=x \in [0,1]$. By Denisov's decomposition (Theorem \ref{DenisovBM}), for all $f: \mathcal{C}[0,\infty) \rightarrow \mathbb{R}$ bounded and continuous, 
\begin{equation*} 
\mathbb{E}^{\bf W} f(\alpha_t; t \geq 0) = \int_0^1 \frac{dx}{ \pi \sqrt{ x (1-x)}} \mathbb{E} f(\alpha^x_t; t \geq 0).
\end{equation*}
where $\mathbb{E}^{\bf W}$ is the expectation relative to ${\bf W}$.

The Feller property of $(\alpha_t; t \geq 0)$ follows from a direct computation of the transition semigroup $Q_t(x,\cdot)$ of $(\alpha^x_t; t \geq 0)$, which will be given in Section \ref{s35}. But here we provide a conceptual proof.

\begin{proposition}
 \label{Fellerconcept}
The argmin process $(\alpha_t; t \geq 0)$ enjoys the Feller property, and is a strong Markov process.
\end{proposition}
\begin{proof}
 According to Kallenberg \cite[Lemma 19.3]{Kallenberg}, it suffices to show that
\begin{enumerate}
\item
for each $t \geq 0$, $\alpha_t^x \rightarrow \alpha_t^y$ in distribution as $x \rightarrow y$;
\item
for each $x \in [0,1]$, $\alpha_t^x \rightarrow x$ in probability as $t \rightarrow 0$.
\end{enumerate}
We first prove $(1)$. For $t \geq 1$, $\alpha_t^x$ and $\alpha_t^y$ are both arcsine distributed regardless of $x,y \in [0,1]$. Consider the case $t \in (0,1)$. By Denisov's decomposition (Theorem \ref{DenisovBM}), for all $f: \mathcal{C}[0,2] \rightarrow \mathbb{R}$ bounded and continuous,
\begin{equation*}
\mathbb{E}^{\bf W} f(B_u; 0 \leq u \leq 2) = \int_0^1 \frac{dx}{\pi \sqrt{x(1-x)}} \mathbb{E}^{{\bf P}^x} f(B_u; 0 \leq u \leq 2).
\end{equation*}
By the explicit scaling construction of Brownian meanders, the law of $(B_u; 0 \leq u \leq 2)$ under ${\bf P}^x$ is weakly continuous in $x$. Moreover, for each $t \in (0,1)$, $\alpha_t(w)$ is an a.s. continuous functional of $(w_u; 0 \leq u \leq 2)$, from which follows $(1)$.

It remains to prove $(2)$. Observe that for $t<1$, $1-t \leq \alpha_t^1 \leq 1$. So $(2)$ is proved in case of $x=1$. For $x \in [0,1)$, by Denisov's decomposition (Theorem \ref{DenisovBM}),
\begin{equation*}
\inf\{u \geq 1; w_u \leq w_x\}>1 \quad {\bf P}^x\mbox{-}a.s.
\end{equation*}
Therefore, $\mathbb{P}(\alpha_t^x=x-t~ \mbox{for}~t~\mbox{close to}~ 0)=1$, which leads to the desired result.
\end{proof}
\subsection{Jumps and L\'evy system}
\label{s34}
We study the jumps of the argmin process $\alpha$. Recall the definitions of a $\frac{1}{2}-$stable subordinator, and the age process of a $\frac{1}{2}$-stable subordinator from \eqref{halfss} and \eqref{age}. We begin with the following observation.
\begin{lemma}
\label{lct}
Let $\rho: =\inf\{t>0; \alpha_t=1\}$ and $\tau:=\inf\{t>\rho; \alpha_t=0\}$. Then $(1-\alpha_{\rho+t}; 0 \leq t \leq \tau - \rho)$ has the same distribution as the age process of a $\frac{1}{2}-$stable subordinator until it first reaches $1$.
\end{lemma}
\begin{proof}
By the strong Markov property of Brownian motion, $(B_{\rho+1+u}-B_{\rho+1}; u \geq 0)$ is still Brownian motion. It is not hard to see that $(1-\alpha_{\rho+t}; 0 \leq t \leq \tau - \rho)$ is the age process derived from excursions above the past minimum of $(B_{\rho+1+u}-B_{\rho+1}; u \geq 0)$ until this post$-(\rho+1)$ Brownian motion escapes its past minimum by time $1$. This yields the desired result.
\end{proof}

By Lemma \ref{lct}, let $(l^1_t; t \geq 0)$ be the local times of $\alpha$ at level $1$, normalized to match the $\frac{1}{2}-$stable subordinator. By time-reversal of $\alpha$ (Proposition \ref{timereversal}), define similarly $(l^0_t; t \geq 0)$ to be the local times of $\alpha$ at level $0$. By stationarity of $\alpha$,  
\begin{equation*}
\mathbb{E}l^1_t / t  = \mathbb{E}l^0_t / t = c \quad \mbox{for all}~t > 0.
\end{equation*}
We will prove in Corollary \ref{lctrate} that the constant $c=1/\sqrt{2 \pi}$. These stationary local times also appeared in the work of Leuridan \cite{Leuridan}. 

Before proceeding further, we need the following terminology. Let $(X_t; t \geq 0)$ be a Hunt process on a suitably nice state space $E$, e.g. locally compact and separable metric space. The pair $(\Pi,C)$ constituted of a kernel $\Pi$ on $E$ and a continuous additive functional $C$ is said to be a {\em L\'evy system} for $X$ if for all bounded and measurable function $f$ on $E \times E$,
\begin{equation}
\label{Levysystemdef}
\mathbb{E}\left(\sum_{0<s \leq t} f(X_{s-},X_s)1_{\{X_{s-} \neq X_s\}}\right) = \mathbb{E}\left(\int_0^t dC_s \int_E \Pi(X_{s-},dy) f(X_{s-},y)\right).
\end{equation}
The kernel $\Pi$ is called the {\em L\'evy measure} of the additive functional $C$. The notion of a L\'evy system was formulated by Watanabe \cite{Wata}, the existence of which was proved for a Hunt process under additional assumptions. The proof was simplified by Beneviste and Jacod \cite{BJ}. See also Meyer \cite{Meyer71}, Pitman \cite{PitmanLevy} and Sharpe \cite[Chapter VIII]{Sharpe} for development.

By Proposition \ref{Fellerconcept}, the argmin process $(\alpha_t; t \geq 0)$ is a Hunt process. Also define a continuous additive functional $C$ by
\begin{equation}
\label{argminaf}
C_t = t + l_t^0 \quad \mbox{for}~t \geq 0.
\end{equation}
The main result is stated as follows, the proof of which relies on Lemmas \ref{jr1} and \ref{jr0}.
\begin{theorem}
\label{argminls}
Let $(\alpha_t; t \geq 0)$ be the argmin process of Brownian motion, and $(C_t; t \geq 0)$ be the additive functional as in \eqref{argminaf}. Define a kernel $\Pi$ on $[0,1]$ by
\begin{equation}
\label{argminlm}
\Pi(x,dy) =     \left\{ \begin{array}{ccl}
         \Pi^{0 \uparrow}(dy) & \mbox{for} & x=0,\\
         \mu^{\uparrow 1}(x) \delta_1 & \mbox{for}
         & x \in (0,1), \\
         0 & \mbox{for} & x=1.
                \end{array}\right.
\end{equation}
where $\delta_1$ is the point mass at $1$,
\begin{equation}
\label{jrt0}
\Pi^{0 \uparrow}(dy): = \frac{dy}{\sqrt{2 \pi y^3 (1-y)}}  \quad \mbox{for}~0<y<1,
\end{equation}
and 
\begin{equation}
\label{jrt1}
\mu^{\uparrow 1}(x): = \frac{1}{2(1-x)} \quad \mbox{for}~0<x<1.
\end{equation}
Then $(C,\Pi)$ is a L\'evy system of $(\alpha_t; t \geq 0)$.
\end{theorem}

Recall from Proposition \ref{argminBMsp} \eqref{jumptypb} that $(\alpha_t; t \geq 0)$ can only have $(i).$ jumps from $0$ to some $x \in (0,1)$, and $(ii).$ jumps from some $x \in (0,1)$ to $1$. We start by computing the jump rate of $\alpha$ from $x \in (0,1)$ to 1.

\begin{lemma}
\label{jr1}
Let $(\alpha_t; t \geq 0)$ be the argmin process of Brownian motion. 
\begin{enumerate}
\item
Let $x > y \geq 0$. The probability that $\alpha^x$ decreases at unit speed from $x$ to $y$ with no jumps is given by
\begin{equation}
\label{sxy}
s(x,y)  = \sqrt{\frac{1-x}{1-y}}.
\end{equation}
\item
For $x \in (0,1)$, the jump rate of $\alpha$ per unit time from $x$ to $1$ is $\mu^{\uparrow 1}(x)$ defined by \eqref{jrt1}.
\end{enumerate}
\end{lemma}
\begin{proof}
$(1)$ Let $\widetilde{x}:=1-x$ and $\widetilde{y}:=1-y$. By Denisov's decomposition (Theorem \ref{DenisovBM}), under ${\bf P}^x$, $(B_{x+t} - B_x; 0 \leq t \leq \widetilde{x})$ is a Brownian meander of length $\widetilde{x}$, independent of Brownian motion $(B_{1+t} - B_1; t \geq 0)$. 

Let $Z_{x-y}$ be normally distributed with mean $0$ and variance $x-y$. Let $R^{\widetilde{x}}$ be Rayleigh distributed with parameter $\widetilde{x}$, independent of $Z_{x-y}$. We have
\begin{align*}
    s(x,y) & = {\bf P}^x\left(\inf_{t \leq x-y} (B_{1+t} - B_1) > B_x - B_1 \right) \\
           & =\mathbb{P}(|Z_{x-y}| < R^{\widetilde{x}})\\
           & = 2 \int_0^{\infty} \mathbb{P}(R^{\widetilde{x}} > z) \cdot \frac{1}{\sqrt{2 \pi (x-y)}} \exp\left(-\frac{z^2}{2(x-y)} \right) dz = \sqrt{\frac{\widetilde{x}}{\widetilde{y}}},
\end{align*}
where the second equality follows from the reflection principle of Brownian motion, and the fact that a Brownian meander of length $\widetilde{x}$ evaluated at time $\widetilde{x}$ is Rayleigh distributed with parameter $\widetilde{x}$, whose density is given by \eqref{meanderRay}.

$(2)$ Note that   
\begin{equation*}
    s(x,y) : = \exp \left(-\int_y^x dz \mu^{\uparrow 1}(z) \right).
\end{equation*}
We obtain the jump rate \eqref{jrt1} by taking derivative of \eqref{sxy} with respect to $x$.
 \end{proof}

\begin{remark}
\label{jremark}
{\em We provide an alternative approach to Lemma \ref{jr1}. Consider the excursions above the past minimum of $(B_t; t \geq 0)$. Given $\alpha_0=x$, it must be a ladder time; that is the starting time of an excursion. Thus, the probability that $\alpha$ jumps to $1$ on $(0,dt]$ given $\alpha_0=x$ is the same as that of an excursion terminates in $dt$ given that it has reached length $\widetilde{x}$.

Let $\zeta$ be the length of such an excursion. By \eqref{Itolength}, the aforementioned probability is given by 
\begin{equation*}
{\bf n}(\zeta \in \widetilde{x}+dt | \zeta>\widetilde{x}) = \frac{\Lambda(d \widetilde{x})/d\widetilde{x}}{\Lambda(\widetilde{x}, \infty)}dt =  \frac{1}{2 \widetilde{x}} dt,
\end{equation*}
where $\Lambda(dx)$ is the L\'evy measure of a $\frac{1}{2}$-stable subordinator as in \eqref{halfss}.
This gives the jump rate \eqref{jrt1}.
}
\end{remark}

To conclude this subsection, we compute the L\'{e}vy measure of jumps of $\alpha$ in from $0$.
\begin{lemma}
\label{jr0}
Let $(\alpha_t; t \geq 0)$ be the argmin process of Brownian motion. For $y \in (0,1)$, the L\'{e}vy measure of jumps of $\alpha$ per unit local time in from $0$ is $\Pi^{0 \uparrow}(dy)$ defined by \eqref{jrt0}.
\end{lemma}
\begin{proof}
On one hand, the mean number of jumps per unit time from $0$ to $dy$ near $y$ is $\Pi^{0 \uparrow} (dy) \mathbb{E}l^0_1 = \Pi^{0 \uparrow} (dy) / \sqrt{2 \pi}$. On the other hand, the mean number of jumps per unit time from $dy$ near $1-y$ to $1$ is given by
$$\frac{dy}{\pi \sqrt{y (1-y)}} \mu^{\uparrow 1}(1-y) = \frac{dy}{2 \pi \sqrt{y^3 (1-y)}}.$$
By Proposition \ref{timereversal}, we identify these two quantities and obtain the L\'{e}vy measure \eqref{jrt0}.
\end{proof}
\subsection{Transition kernel}
\label{s35}
We complete the proof of Theorem \ref{Ftrans}. Recall the definition of $(\alpha_t^x; t \geq 0)$ from \eqref{argmincd}, which is viewed as the argmin process $\alpha$ conditioned on $\alpha_0 = x$. For $0 \leq b \leq 1$, let 
$$\tau_b^x : = \inf\{t>0; \alpha_t^x = b\} $$
to be the first time at which $(\alpha_t^x; t \geq 0)$ hits level $b$. Also recall the definition of $\mu^{\uparrow 1}(x)$ from \eqref{jrt1}. We start with a lemma whose proof is straightforward.
\begin{lemma}
\label{interh}
For $0 < x <1$,
\begin{equation}
\label{ttt}
\mathbb{P}(\tau_1^x \in dt) = \mu^{\uparrow 1}(x-t) s(x,x-t) dt \quad \mbox{if}~0 < t <x,
\end{equation}
where $\mu^{\uparrow 1}(x)$ is given by \eqref{jrt1} and $s(x,y)$ is given by \eqref{sxy}.
\end{lemma}

\begin{proof}[Proof of Theorem \ref{Ftrans}]
The first part of Theorem \ref{Ftrans} has been proved as Proposition \ref{Fellerconcept}. Now we compute the transition kernel $Q_t(x,dy)$ for $t >0$ and $x \in [0,1]$ of $(\alpha_t; t \geq 0)$. 

Observe that $\alpha_{t+s}$ and $\alpha_s$ are independent for all $t \geq 1$. By Proposition \ref{stationary},
\begin{equation}
Q_t(x, dy) = \frac{1_{\{0<y<1\}}}{ \pi \sqrt{y(1-y)}}dy  \quad \mbox{for}~t \geq 1~\mbox{and}~x \in [0,1],
\end{equation}
which is the invariant measure of the argmin process $\alpha$.

Given $\alpha_0=1$, we have $B_u \geq B_1$ for all $u \in [0,1]$. So for $0 < t \leq 1$,
$$ t + \alpha^1_t\stackrel{(d)}{=} \sup \left\{s \in [1,1+t]; B_s = \max_{u \in [1,1+t]} B_u \right\}.$$
Consequently, $\alpha^1_t$ is the arcsine distribution rescaled linearly into $[1-t, 1]$. That is, 
\begin{equation}
\label{Case2}
Q_t(1,dy)=  \frac{1_{\{1-t < y < 1\}}}{\pi \sqrt{(1-y)(y+t-1)}} dy \quad \mbox{for}~0<t \leq 1.
\end{equation}
By conditioning on $\tau^x_1$ with $\tau^x_1 \leq x$, we have for $0 < t \leq x \leq 1$,
\begin{equation}
\label{Case3}
Q_t(x, dy) = s(x,x-t) \delta_{x-t}(dy) + \int_{x-t}^{x} dz \mu^{\uparrow 1}(z) s(x,z) Q_{t+z-x}(1,dy),
\end{equation}
while for $0 <x < t  \leq 1$,
\begin{equation}
\label{Case4}
Q_t(x,dy) = s(x,0) Q_{t-x}(0,dy) + \int_0^x dz \mu^{\uparrow 1}(z) s(x,z) Q_{t+z-x}(1,dy).
\end{equation}
\quad In the case $t \leq x$ there is an atom of probability $s(x,x-t)$ at $x-t$, whereas in the case $t>x$ this atom is replaced by probability $s(x,0)$ redistributed according to $Q_{t-x}(0,dy)$. For $t=1$, we know that $Q_1(x, dy)$ is arcsine distributed, whatever $x$. So this case gives a formula for $Q_u(0,dy)$ for any $0<u<1$ with $u:= 1-x$. That is,
\begin{equation}
\label{Case5}
Q_u(0,dy) = \frac{1}{s(1-u, 0)} \left[ \frac{1_{\{0<y<1\}}dy}{\pi \sqrt{y(1-y)}} - \int_0^{1-u} dz  \mu^{\uparrow 1}(z) s(1-u,z) Q_{u+z}(1,dy)  \right].
\end{equation}
It remains to evaluate the r.h.s. of \eqref{Case3}-\eqref{Case5}. By \eqref{sxy} and \eqref{Case2}, we get
$$s(x,x-t) = \sqrt{\frac{1-x}{1-x+t}},$$
and 
$$
\int_{x-t}^{x} dz \mu^{\uparrow 1}(z) s(x,z) Q_{t+z-x}(1,dy)  =  \frac{\sqrt{(y+t-1)^{+}}}{\pi (y+t-x)\sqrt{1-y}} dy .
$$
By injecting these expressions into \eqref{Case3}, we obtain
\begin{equation}
\label{tx1}
Q_t(x, dy)  = \sqrt{\frac{1-x}{1-x+t}} \delta_{x-t}(dy) + \frac{\sqrt{(y+t-1)^{+}}}{\pi (y+t-x) \sqrt{1-y}} dy \quad \mbox{for}~t \leq x \leq 1.
\end{equation}
Similarly, we get from \eqref{Case4} that for $x<t \leq 1$,
\begin{equation*}
Q_t(x,dy) = \sqrt{1-x} Q_{t-x}(0,dy) + \frac{dy \sqrt{1-x}}{\pi (y+t-x) \sqrt{1-y}} \left[\sqrt{\frac{(y+t-1)^{+}}{1-x}} - \sqrt{(y+t-1-x)^{+}} \right];
\end{equation*}
and from \eqref{Case5} that for $0<u<1$,
\begin{equation*}
Q_u(0,dy) = \frac{\sqrt{u} + \sqrt{y(y+u-1)^{+}}}{\pi (y+u) \sqrt{y(1-y)}} dy.
\end{equation*}
By combining the above expressions, we obtain
\begin{equation}
\label{xt1}
Q_t(x,dy) = \frac{\sqrt{(1-x)(t-x)} + \sqrt{y(y+t-1)^{+}}}{\pi (y+t-x) \sqrt{y(1-y)}} dy \quad \mbox{for}~x<t \leq 1.
\end{equation}
\end{proof}
\subsection{Breakdown of Dynkin's and Rogers-Pitman criterion}
\label{s32}
In this part, we explain why Dynkin's criterion, and the Rogers-Pitman criterion fail to prove that $(\alpha_t; t \geq 0)$ is Markov. Before proceeding further, we recall these sufficient conditions for a function of a Markov process to be Markov.

Let $(X_t; t \geq 0)$ be a continuous-time Markov process on a measurable state space $(E, \mathcal{E})$, with initial distribution $\lambda$ and transition semigroup $(P_t; t \geq 0)$. Let $(E',\mathcal{E}')$ be a second measurable space, and $\phi: (E,\mathcal{E}) \rightarrow (E',\mathcal{E}')$ be a measurable function.

Dynkin \cite{Dynkin} initiated the study of Markov functions, and gave a condition for $(\phi (X_t); t \geq 0)$ to be Markov for all initial distributions $\lambda$. Later Rogers and Pitman \cite{RogersPitman} made a simple observation: if there exists a Markov kernel $\Lambda: E' \times \mathcal{E} \ni (y, A) \mapsto \Lambda(y,A) \in \mathbb{R}_{+}$ such that for all $t \geq 0$ and $A \in \mathcal{E}$,
\begin{equation}
\label{inversef}
\mathbb{P}(X_t \in A | \phi(X_s), 0 \leq s \leq t) = \Lambda\left(\phi(X_t), A\right)\quad a.s.,
\end{equation}
then $(\phi(X_t); t \geq 0)$ is Markov with transition kernels 
$$Q_t = \Lambda P_t \Phi \quad \mbox{for all}~ t \geq 0,$$
where $\Phi$ is the Markov kernel from $E'$ to $E$ induced by $\phi$: 
$\Phi(x, B) = \delta_{\phi(x)}(B)$ for $x \in E$ and $B \in \mathcal{E}'$. 
The following theorem provides a sufficient condition for \eqref{inversef} to hold.
\begin{theorem}[Rogers-Pitman criterion] \cite{RogersPitman}
\label{RogPit}
Let $\Phi$ be derived from $\phi: E \rightarrow E'$ as above. Assume that there exists a Markov kernel $\Lambda$ from $E'$ to $E$ such that 
\begin{enumerate}
\item[(i).]
$\Lambda \Phi = I$, the identity kernel;
\item[(ii).] \label{intertw}
for each $t \geq 0$, the Markov kernel $Q_t = \Lambda P_t \Phi$ satisfies the intertwining relation $\Lambda P_t = Q_t \Lambda$.
\end{enumerate}
Let $(X_t); t \geq 0$ be Markov with initial distribution $\lambda = \Lambda(y, \cdot)$ for some $y \in E'$ and semigroup $(P_t; t \geq 0)$.
Then \eqref{inversef} holds, and $(\phi(X_t); t \geq 0)$ is Markov with starting state $y$ and transition semigroup $(Q_t; t \geq 0)$.
\end{theorem}
Note that if instead of $(ii)$, 
$$P_t \Phi = \Phi Q_t  \quad \mbox{for all}~t \geq 0,$$
for a Markov kernel $Q_t$ on $E'$, then $(\phi(X_t); t \geq 0)$ is Markov for all initial distributions $\lambda$. This recovers Dynkin's criterion \cite{Dynkin}. 

As shown by \eqref{funcm}, the argmin process $\alpha$ is a measurable function of the space-time shift process $(\Theta_t; t \geq 0)$ whose transition kernel is given by 
\begin{equation}
\label{mwsg}
P_t(w,d\widetilde{w}) =   \delta_{\Theta_t w}(d\widetilde{w}),
\end{equation}
where $\Theta_t w : = (w_{t+u} - w_t; t \geq 0)$ is the space-time shift of $w$ on $\mathcal{C}[0,\infty)$. 
For $w \in \mathcal{C}[0,\infty)$ and $t \geq 0$, let
\begin{equation}
\label{alphatwhaha}
\alpha_t(w) : = \sup\left\{s \in [0,1]: w_{t+s} = \inf_{u \in [0,1]} w_{t+u}\right\}.
\end{equation}
The Markov kernel $\Phi$ induced by the argmin function is given by
\begin{equation}
\label{phihaha}
\Phi(w,\cdot) = \delta_{\alpha_0(w)}(\cdot) \quad \mbox{for all}~w \in \mathcal{C}[0,\infty).
\end{equation}
We first show that Dynkin's criterion does not hold.
\begin{proposition}
Let $(P_t; t \geq 0)$ be the semigroup of the space-time shift process $\Theta$ given by \eqref{mwsg}, $(Q_t; t \geq 0)$ be the semigroup of the argmin process $\alpha$ given by \eqref{Qtrans}, and $\Phi$ be the Markov kernel defined by \eqref{phihaha}. Then
\begin{equation*}
    P_t \Phi (w,\cdot) \neq \Phi Q_t (w,\cdot) \quad \mbox{for all}~t>0.
\end{equation*}
\end{proposition}
\begin{proof}
Observe that for all $t \geq 0$,
\begin{equation*}
    P_t \Phi (w,\cdot) = \delta_{\alpha_t(w)}(\cdot) \quad \mbox{and} \quad \Phi Q_t (w,\cdot) = Q_t(\alpha_0(w),\cdot),
\end{equation*}
where $\alpha_t(w)$ is given by \eqref{alphatwhaha}. From this follows the result.
\end{proof}

Recall the definition of ${\bf P}^x$ from \eqref{defPx}. By Denisov's decomposition (Theorem \ref{DenisovBM}), the condition \eqref{inversef} amounts to
\begin{equation}
\label{kernelmd}
\Lambda(x , \cdot) = {\bf P}^x \quad \mbox{for all}~x \in [0,1].
\end{equation}
The following result shows that Rogers-Pitman intertwining criterion does not hold.
\begin{proposition}
\label{notRP}
Let $(P_t; t \geq 0)$ be the semigroup of the space-time shift process $\Theta$ given by \eqref{mwsg}, $(Q_t; t \geq 0)$ be the semigroup of the argmin process $\alpha$ given by \eqref{Qtrans}, and $\Lambda$ be the Markov kernel defined by \eqref{kernelmd}.
Then for each $t \in (0,1]$,
$$\Lambda P_t(t,\cdot)  \neq Q_t \Lambda(t,\cdot).$$
\end{proposition}
\begin{proof}
Observe that $\Lambda P_t(t,\cdot) = \overset{\longrightarrow}{{\bf M}^{1-t}} \otimes {\bf W}$: the law of a Brownian meander of length $1-t$ concatenated by Brownian motion. Thus,
$$\mathbb{E}^{\Lambda P_t(t,\cdot)}[w_1] = \sqrt{\frac{\pi}{2}(1-t)}.$$
Next by \eqref{Qtrans},
$$Q_t \Lambda(t,\cdot) =  \int Q_t(t,dy) {\bf P}^y = \sqrt{1-t} {\bf P}^0 + \int_{1-t}^1\frac{\sqrt{y+t-1}}{\pi y \sqrt{1-y}}{\bf P}^y dy,$$
which implies that
\begin{align*}
\mathbb{E}^{Q_t \Lambda(t,\cdot)}[w_1] &=  \sqrt{\frac{\pi}{2}(1-t)} +\frac{1}{\sqrt{2 \pi}} \int_{1-t}^1 \frac{\sqrt{y+t-1}}{\sqrt{y(1-y)}} dy \\
 & > \mathbb{E}^{\Lambda P_t(t,\cdot)}[w_1] \quad \mbox{for}~t \in (0.1].
 \end{align*}
This yields the desired result.
\end{proof}

\section{The $(a,b)$-minima set of Brownian motion}
\label{s4}
In this section, we study the $(a,b)$-minima set $\mathcal{M}_{a,b}$ of Brownian motion defined by \eqref{abminima}. In Section \ref{s42}, we consider the renewal property of the set $\mathcal{M}_{a,b}$, and provide an alternative proof of Theorem \ref{Leust}.
In Section \ref{s43}, we give an explicit construction for times of the set $\mathcal{M}_{1,1}$, which implies Theorem \ref{minimabetlaw}.
Finally in Section \ref{s41}, we deal with the sample path of Brownian motion between two $(1,1)$-minima. There Theorem \ref{ABC} is proved. 
\subsection{Renewal structure of $(a,b)$-minima}
\label{s42}
We provide an alternative proof of Theorem \ref{Leust}. Recall that the argmin process $\alpha$ is a stationary Markov process, whose
\begin{itemize}
\item 
invariant measure has density $f(x)$, $0 <x <1$ given by \eqref{arcsinesta};
\item
transition kernel $Q_t(x,\cdot)$, $t>0$ and $x \in [0,1]$ is given by \eqref{Qtrans}.
\end{itemize}
For $y \neq x-t$, write $Q_t(x,dy) = q_t(x,y) dy$.
The following lemma, which is crucial in Leuridan's proof of Theorem \ref{Leust}, can be derived from Proposition \ref{stationary} and Theorem \ref{Ftrans}.

\begin{lemma} \cite{NeveuPitman, Leuridan}
\label{Gapmeanlem}
Given a measurable set $A \subset \mathbb{R}^{+}$, let $N_{a,b}(A): = \#(\mathcal{M}_{a,b} \cap A)$ be the counting measure of $(a,b)$-minima. Then 
\begin{equation}
\label{mean}
\mathbb{E}(N_{a,b}(dt)) = \frac{dt}{\pi \sqrt{ab}}
\end{equation}
and for $0 \leq s < t$,
\begin{equation}
\label{correlation}
\mathbb{E}(N_{a,b}(ds)N_{a,b}(dt)) = \frac{1}{\pi \sqrt{ab}}h_{a,b}(t-s)dsdt,
\end{equation}
where $h_{a,b}$ is defined by \eqref{hab}.
\end{lemma}

\begin{proof}
Observe that 
$$N_{a,b}(dt) =1 \Longleftrightarrow \mbox{the minimum of}~B~\mbox{on}~[t-a,t+b]~\mbox{is achieved in}~dt.$$
By Brownian scaling, the latter has the same probability as that of $\displaystyle \left\{(a+b) \alpha_{\frac{t-a}{a+b}} \in [a, a+dt] \right\}$. So
\begin{align*}
    \mathbb{E}(N_{a,b}(dt)) = \frac{1}{a+b} f\left(\frac{a}{a+b}\right) dt  = \frac{dt}{\pi \sqrt{ab}}.
\end{align*}
A similar argument shows that
\begin{align*}
    \mathbb{E}(N_{a,b}(ds)N_{a,b}(dt)) & = \mathbb{E}(N_{a,b}(ds)) \cdot \frac{1}{a+b} q_{\frac{t-s}{a+b}} \left(\frac{a}{a+b}, \frac{a}{a+b} \right) dt \\
                            & = \frac{1}{\pi \sqrt{ab}} h_{a,b}(t-s) dsdt.
\end{align*}
\end{proof}
 
\begin{proof}[Proof of Theorem \ref{Leust}]
Note that $\mathcal{M}_{a,b} - a : = (T^{a,b}_i-a; i \geq 1)$ is a renewal process with stationary delay. 
By Lemma \ref{Gapmeanlem}, $h_{a,b}(\cdot)$ is the {\em renewal function} of the point process $\mathcal{M}_{a,b} - a$. The formula \eqref{Gaplaw} follows from Daley and Vere-Jones \cite[Example 5.4(b)]{DVJ}. 
By renewal theory, the law of $T_{1}^{a,b} - a$ is obtained first by size-biasing the inter-arrival time distribution \eqref{Gaplaw}, and then by stick-breaking uniformly at random, see Thorisson \cite{Th1995}. This gives the formula \eqref{T1law}.
\end{proof}
\subsection{Construction of $(1,1)$-minima}
\label{s43}
We consider the case $a=b=1$ by studying the law of Brownian fragments between $(1,1)$-minima. Let 
$T_1, T_2, \cdots$ with $0<T_1<T_2<\cdots$ be times of $(1,1)$-minima set of Brownian motion. 
Now we give a path construction for $T_1, T_2, \cdots$, from which the renewal property of $\mathcal{M}_{1,1}$ is clear.
In particular, Theorem \ref{minimabetlaw} is a corollary of this construction.

\vskip 8pt
\fbox{{\bf Construction of $T_1$}}
Let $J$ be the first descending ladder time of $B$, from which starts an excursion above the minimum of length exceeding $1$. The Laplace transform of $J$ is given by \eqref{Laplace}. By Theorem \ref{exmeander}, $(B - \underline{B})[J,J+1]$ is a Brownian meander of length $1$. 

If $J \geq 1$, then $T_1 = J$. If not, we start afresh Brownian motion at the stopping time $J+1$; that is $B^{1}: = (B_{J+1+t} - B_{J+1}; t \geq 0)$. Let $J_1$ be constructed as $J$ for $B^1$. Thus, $J_1 \in \mathcal{LE}$, and $(B^1 - \underline{B}^1)[J_1,J_1+1]$ is a Brownian meander of length $1$. Now we look backward a unit from $J_1$ to see whether $J_1 \in \mathcal{RE}$ or not. If $J_{1} \in \mathcal{RE}$, then $T_1 = J_1$. If not, we start afresh Brownian motion at $J_1+1$ and proceed as before until a $(1,1)$-minima is found.

\vskip 8 pt
\fbox{{\bf Construction of $T_{i+1}$ given $T_i$}} By induction, $(B_{T_i+t} - B_{T_i}; 0 \leq t \leq 1)$ is a Brownian meander of length $1$. Now it suffices to start afresh Brownian motion at $T_{i}+1$, and proceed as in the construction of $T_1$. 

\vskip 8 pt
\fbox{\bf Evaluation of the geometric rate} Recall that $\Delta$ is distributed as $T_{i+1} - T_i$. Let 
\begin{equation}
\label{Ng}
N: = \inf\{i \geq 1: J_i \in \mathcal{RE}\}.
\end{equation}
It is easy to see that $N$ is geometrically distributed on $\{1,2,\cdots\}$ with parameter $\mathbb{P}(J_1 \in \mathcal{RE})$. Note that $J_i$ depends on the event $\{N = i\}$, but is independent of the event $\{N \geq i\}$. In fact, $N$ is a stopping time of a sequence of i.i.d. path fragments, each starting with a meander and continuing with an independent Brownian motion until time $J_i$. By Wald's identity,
$$\mathbb{E}\Delta = \mathbb{E}N \cdot (1 + \mathbb{E}J_1) = \frac{1 + \mathbb{E}J_1}{\mathbb{P}(J_1 \in \mathcal{RE})}.$$
Now by \eqref{Gapmean}, we get
\begin{equation}
\label{qval}
 \mathbb{P}(J_1 \in \mathcal{RE}) = \frac{2}{\pi}.
\end{equation}

In view of the dependence of $J_i$ and the event $\{N=i\}$, the evaluation of the geometric rate in the distribution of $N$ is quite indirect. Here is a more direct approach.

Consider the construction of $J_1$ as $J$ for a copy of Brownian motion preceded by an independent meander of length $1$.
It is straightforward that
\begin{equation}
\label{strai}
\mathbb{P}(J_1 \in \mathcal{RE}~\mbox{and}~J_1 \geq  1) =  P(J_1 \geq  1 ) =  1 -\frac{2}{\pi},
\end{equation}
where the second equality is obtained by integrating \eqref{firstlaw} over $[0,1]$.
The evaluation of $\mathbb{P}(J_1 \in \mathcal{RE}~\mbox{and}~J_1 < 1)$ is more tricky, which relies on the following lemma. 
\begin{lemma}
Let $(L^{br}_t; 0 \leq t \leq 1)$ be the local time process of a Brownian bridge of length $1$ at level $0$. Then
\begin{equation}
\label{striking}
\mathbb{P}(L_t^{br}> x) = e^{-\frac{x^2}{2}} \mathbb{P}\left(|B_1| > x \sqrt{\frac{1-t}{t}} \right)  \quad \mbox{for}~t \in [0,1], x>0.
\end{equation}
\end{lemma}
\begin{proof}
 It can be read from Pitman \cite[(3)]{Pitcal} that for $t \in [0,1]$, $x>0$ and $y \in \mathbb{R}$,
\begin{equation}
\label{Pitbg}
\mathbb{P}(L_t^{br} >x| B_t^{br} \in dy) = \exp\left(-\frac{1}{2}\left[ \left(\frac{|y|}{t} + \frac{x}{t} \right)^2 - \frac{y^2}{t} \right] \right).
\end{equation}
By integrating \eqref{Pitbg} with respect to the normal density of $B_t^{br}$ with mean $0$ and variance $t(1-t)$, we obtain \eqref{striking}.
\end{proof}
Let $- \xi$ be the level of the minimum of the free Brownian part of the path at time $J_1$ so that
$J_1 = \sigma_\xi$, where $\sigma$ is $\frac{1}{2}$-stable surbordinator with jumps of size larger than $1$ deleted. Recall from Section \ref{s22} that $\xi$ is exponentially distributed with parameter $\sqrt{2 /\pi}$. By letting $T_x: = \inf\{t>0: B_t = x\}$, we obtain for $0<t<1$,
\begin{align}
\label{jointL1}
\mathbb{P}(\xi \in dx,J_1 \in dt) & = \sqrt{2/\pi} dx \mathbb{P}(T_x \in dt) \notag\\
                                        & = \frac{x}{\pi t^{3/2}} e^{-\frac{x^2}{2t}}dtdx.
\end{align}
By time-reversing the Biane-Yor construction \cite{BY} of Brownian meander minus its future minimum process (see also Bertoin and Pitman \cite[Theorem 3.1]{BP}), we get
\begin{align}
\label{BYcal}
\mathbb{P}(J_1 \notin \mathcal{RE}~\mbox{and}~J_1<1) & = \int_{0}^{\infty} \int_{0}^1  \mathbb{P}(L_{1-t}^{br}>x) \mathbb{P}(J_1 \in dt,\xi \in dx) \notag\\
              & = 1 -\frac{2}{\pi},
\end{align}
where the last equality is obtained by plugging in \eqref{striking} and \eqref{jointL1}. Now \eqref{qval} follows readily from \eqref{strai} and \eqref{BYcal}.

\begin{proof}[Proof of Theorem \ref{minimabetlaw}]
For any random variable $X$, let $\Phi_X(\lambda)$ be the Laplace transform of $X$. The identity \eqref{001} is clear from the preceding construction. It implies that 
\begin{align*}
\mathbb{E}T_1 & = \mathbb{E}J+ \mathbb{P}(J<1) \mathbb{E}\Delta \\
              & = 1 + \frac{2}{\pi} \cdot \pi = 3,
\end{align*}
where the second equality follows from \eqref{meanfirst}, \eqref{firstlaw}, and \eqref{Gapmean}. In addition,
\begin{equation}
\label{001Lap}
\Phi_{T_1}(\lambda) = \mathbb{E}(e^{-\lambda J} 1_{\{J \geq 1\}}) + \Phi_{\Delta}(\lambda) \mathbb{E}(e^{-\lambda J} 1_{\{J<1\}}).
\end{equation}
By integrating with respect to \eqref{firstlaw}, we get
\begin{equation}
\label{001Lap2}
\mathbb{E}(e^{-\lambda J} 1_{\{J <1\}}) = \frac{\erf(\sqrt{\lambda})}{\sqrt{\pi \lambda}}.
\end{equation}
By injecting \eqref{001Lap2} into \eqref{001Lap}, we obtain
\begin{equation}
\label{001Lap3}
\Phi_{T_1}(\lambda) = \Phi_{J}(\lambda) -  \frac{\erf(\sqrt{\lambda})}{\sqrt{\pi \lambda}}(1 -\Phi_{\Delta}(\lambda) ).
\end{equation}
Recall from \eqref{stadelay} that $T_1 - 1$ is the stationary delay for a renewal process with inter-arrival time distributed according to $\Delta$. By renewal theory,
\begin{equation}
\mathbb{P}(T_1 - 1 \in dt)/dt = \frac{1}{\pi} \mathbb{P}(\Delta > t),
\end{equation}
which implies that
\begin{equation}
\label{001Lap4}
\Phi_{T_1}(\lambda) = \frac{e^{-\lambda}}{\pi \lambda}(1 - \Phi_{\Delta}(\lambda)).
\end{equation}
Combining \eqref{001Lap3} and \eqref{001Lap4} yields
\begin{equation}
\label{001Lap5}
\Phi_{T_1}(\lambda) =  e^{-\lambda}(\Phi_{J}(\lambda))^2, 
\end{equation}
and
\begin{equation}
\label{001Lap6}
\Phi_{\Delta}(\lambda) = 1 - \pi \lambda (\Phi_{J}(\lambda))^2,
\end{equation}
Now the identity \eqref{002} follows readily from \eqref{001Lap5}. By plugging the formula \eqref{Laplace} for $\Phi_{J}(\lambda)$ into \eqref{001Lap5} and \eqref{001Lap6}, we get \eqref{Laplace2} and \eqref{Laplace3}.

Let $H$ be distributed as $B_{T_{i+1}} - B_{T_i}$, $i \geq 1$. Recall from \eqref{meanderRay} that a Brownian meander evaluated at time $1$ has Rayleigh distribution with parameter $1$.
It is clear from the above construction that
\begin{equation}
\label{interpath}
H \stackrel{(d)}{=} \sum_{i=1}^N (R_i - \xi_i),
\end{equation}
where $(R_i)_{i \geq 1}$ are i.i.d. Rayleigh distributed with parameter $1$, and $(\xi_i)_{i \geq 1}$ are i.i.d. exponentially distributed with rate $\sqrt{2/\pi}$, independent of $(R_i)_{i \geq 1}$. 
By Wald's identities,
\begin{equation*}
\mathbb{E}H = \mathbb{E}N \cdot (\mathbb{E}R_1 - \mathbb{E}\xi_1) = \frac{\pi}{2} \left( \frac{\pi}{2} - \frac{\pi}{2} \right) = 0,
\end{equation*}
and 
\begin{equation*}
\Var H = \mathbb{E}N \cdot (\Var R_1 + \Var \xi_1) = \frac{\pi}{2} \left(\frac{4-\pi}{2} + \frac{\pi}{2} \right) = \pi.
\end{equation*}
where $\mathbb{E}R_1$ and $\Var R_1$ are given by \eqref{RayEV}. Moreover,
\begin{equation}
\label{1path}
B_{T_{1}}  \stackrel{(d)}{=} -\xi + 1_{\{J<1\}} H,
\end{equation}
with $(\xi, J)$ independent of $H$, and the joint distribution of $(\xi, J)$ given by \eqref{jointL1}.
So
\begin{equation*}
\mathbb{E}B_{T_1} = -\mathbb{E}\xi+ \mathbb{P}(J<1) \cdot \mathbb{E}H=  -\sqrt{\frac{\pi}{2}} + \frac{2}{\pi} \cdot 0 = -\sqrt{\frac{\pi}{2}},
\end{equation*}
and 
\begin{equation*}
    \mathbb{E}B_{T_1}^2 = \mathbb{E}\xi^2  - 2 \mathbb{E}(\xi 1_{\{J<1\}}) \cdot \mathbb{E}H + \mathbb{P}(J<1) \cdot \mathbb{E}H^2 = \pi + 2.
\end{equation*}
\end{proof}

\begin{remark}
\em{By Leuridan's formula \eqref{Gaplaw}, the Laplace transform of $\Delta$ is given by
\begin{equation}
\label{alternat}
\Phi_{\Delta}(\lambda) =  - \sum_{n=1}^{\infty} (-\Psi(\lambda))^n \quad \mbox{for}~\lambda>0~\mbox{such that}~\Psi(\lambda)<1,
\end{equation}
where
\begin{equation} 
\Psi (\lambda) = \frac{2 e^{-\lambda}}{\pi} \int_0^1 \frac{e^{-\lambda t} \sqrt{t}}{t+1}dt + \frac{e^{-2 \lambda}}{\pi \lambda}.
\end{equation}
Since $\Psi(\lambda) \rightarrow 0$ as $\lambda \rightarrow \infty$, $\Psi(\lambda) < 1$ for all sufficiently large $\lambda$. For such large $\lambda$, the expression \eqref{alternat} simplifies to 
\begin{equation}
\label{alternat2}
\Phi_{\Delta}(\lambda) = \frac{\Psi(\lambda)}{1 + \Psi(\lambda)}.
\end{equation}
By analytic continuation, the formula \eqref{alternat2} holds for all $\lambda>0$.
The equality between \eqref{alternat2} and \eqref{Laplace3} reduces to the following identity
\begin{equation}
\label{ident}
\int_0^1 \frac{e^{-\lambda t} \sqrt{t}}{t+1} dt = \frac{\pi e^{\lambda}}{2}\left[ (\erf(\sqrt{\lambda}))^2 - 1 \right] + \sqrt{\frac{\pi}{\lambda}} \erf(\sqrt{\lambda}),
\end{equation}
which can be verified analytically. 
}
\end{remark}

To conclude this part, we give another identity in law similar to \eqref{001}.
\begin{proposition}
Let $U$ be uniform on $[0,1]$, independent of $J$ and $\Delta$. Then we have the following identity in law
\begin{equation}
\label{another}
T_1 - 1 \stackrel{(d)}{=} J + 1_{\{J \le U^2\}} \Delta.
\end{equation}
\end{proposition}
\begin{proof}
Note that $\widetilde{T}_1: = T_1 - 1$ is the stationary delay for a renewal process with inter-arrival time distributed according to $\Delta$. If $J = u > 1$ then $\widetilde{T}_1 = J$, whereas if $J = u < 1$ then
$\widetilde{T}_1 = u$ with probability $\sqrt{u}$, and $\widetilde{T}_1 = u + \Delta$ with probability $1 - \sqrt{u}$.
This is because a meander of length $1$ to the right of time $u$ creates a $(1,1)$-minimum for a two-sided Brownian motion at time $u$ if and only if the meander of length $u$ looking backwards from time $u$ to time $0$ becomes a meander of length $1$ when running further backwards to time $u-1$.
By Brownian excursion theory, the probability that a meander of length $u$ followed by an independent Brownian fragment of length $1-u$ creates a meander of length $1$ is given by
$$\frac{{\bf n}(\zeta>1)}{{\bf n}(\zeta>u)} = \frac{\Lambda(1,\infty)}{\Lambda(u,\infty)} = \sqrt{\frac{2}{\pi}}\Bigg/\sqrt{\frac{2}{\pi u}} = \sqrt{u},$$
where $\Lambda(dx)$ is the L\'evy measure of a $\frac{1}{2}$-stable subordinator defined by \eqref{halfss}. The identity \eqref{another} follows from the above analysis, where $U \sim \mbox{Uniform}[0,1]$ serves as a device to replicate the conditional distribution of $\widetilde{T}_1$ given $J$.
\end{proof}

By conditioning on $J$, the identity \eqref{another} yields a Laplace transform relation, which can be used to provide an alternate derivation of the Laplace transforms of $T_1$ and of $\Delta$.
Though not obviously equivalent, each of the two relations of \eqref{001} and \eqref{another} can be derived from the other after substituting in the explicit formula \eqref{Laplace} for $\Phi_{J}(\lambda)$, and using the simple density of $J$ on $[0,1]$.
However, neither relation seems to offer much insight into their remarkable implication \eqref{002}.
\subsection{A path decomposition between $(1,1)$-minima}
\label{s41}
Let
\begin{equation*}
\mathcal{LE}_b:=\{t \geq 0: B_t < B_s~\mbox{for all}~s \in [t,t+b]\},
\end{equation*}
be the set of left ends of forward meanders of length $b$, and
\begin{equation*}
\mathcal{RE}_a:=\{t \geq a: B_t < B_s~\mbox{for all}~s \in [t-a,t]\},
\end{equation*}
be the set of right ends of backward meanders of length $a$. 
Observe that for $i \geq 1$, $T_i^{a,b} \in \mathcal{LE}_b \cap \mathcal{RE}_a$, and $T_{i+1}^{a,b} = \inf\{t>T_i^{a,b}: t \in \mathcal{LE}_b \cap \mathcal{RE}_a\}$. The following lemma shows that between any $T_i^{a,b}$ and $T_{i+1}^{a,b}$, left ends come before right ends.
\begin{lemma} 
\label{lbefr}
For each $i \geq 1$, let $s \in (T_i^{a,b},T_{i+1}^{a,b}) \cap \mathcal{LE}_b$ and $t \in (T_i^{a,b},T_{i+1}^{a,b}) \cap \mathcal{RE}_a$. Then a.s. $s<t$.
\end{lemma}
\begin{proof}
Suppose by contradiction that there exist $s \in (T_i^{a,b},T_{i+1}^{a,b}) \cap \mathcal{LE}_b$ and $t \in (T_i^{a,b},T_{i+1}^{a,b}) \cap \mathcal{RE}_a$ such that $s \geq t$. Let $r:=\argmin_{u \in [t,s]} B_u$ be the time at which $B$ attains its a.s. unique minimum between $t$ and $s$. It is clear that $r \in \mathcal{LE}_b \cap \mathcal{RE}_a$. Thus, $r \geq T_{i+1}^{a,b}$ by definition of $T_{i+1}^{a,b}$. This is impossible since $r \leq s <T_{i+1}^{a,b}$.
\end{proof}
For $i \geq 1$, let
\begin{equation}
\label{rdef}
D_i^{a,b}: = \inf \{ t>T_i^{a,b}: t \in \mathcal{RE}_a\},
\end{equation}
be the first right end between $T_i^{a,b}$ and $T_{i+1}^{a,b}$, and 
\begin{equation}
\label{ldef}
G_i^{a,b}:=\sup\{t<D_i^{a,b}: t \in \mathcal{LE}_b\}
\end{equation}
be the last left end between $T_i^{a,b}$ and $T_{i+1}^{a,b}$.
Observe that there are neither left ends nor right ends between $G_i^{a,b}$ and $D_i^{a,b}$. By Lemma \ref{lbefr}, the next left end after right ends between $T_{i}^{a,b}$ and $T_{i+1}^{a,b}$ is necessarily a right end; thus is $T_{i+1}^{a,b}$. 
\begin{corollary} 
For each $i \geq 1$, $T_{i+1}^{a,b} = \inf\{t>D_{i}^{a,b}: t \in \mathcal{LE}_b\}$ a.s.
\end{corollary}

From now on, we consider the particular case of $a=b=1$. To simplify notations, write $T_i, \mathcal{LE}, \mathcal{RE}, D_i$ and $G_i$ for $T_i^{1,1}, \mathcal{LE}_1, \mathcal{RE}_1, D_i^{1,1}$ and $G_i^{1,1}$.
The following result characterizes the path fragment $B[G_i,D_i]$.

\begin{proposition}
\label{corodecomp}
Almost surely, for each $i \geq 1$,
\begin{itemize}
\item
$B_{G_i} = B_{D_i}$ and $D_i-G_i > 1$.
\item
$B[G_i, D_i] : = (B_t-B_{G_i}; G_i \leq t \leq D_i)$ consists of two excursions of lengths smaller than $1$.
\end{itemize}
\end{proposition}
\begin{proof}
Suppose by contradiction that $B_{D_i}<B_{G_i}$. Let $D':=\inf\{t>G_i+1: B_t=(B_{D_i}+B_{G_i})/2\}$, and observe that $D' \in \mathcal{RE}$. By path continuity, $D'<D_i$, which contradicts the definition of $D_i$. Similarly, by considering $G'=\sup\{t<D_i-1: B_t=(B_{D_i}+B_{G_i})/2\}$, we exclude the possibility of $B_{D_i}>B_{G_i}$. 

Now we argue by contradiction that there exists $u \in (G_i,D_i)$ such that $B_u < B_{G_i} = B_{D_i}$. Let $M:=\argmin_{u \in [G_i,D_i]} B_u$ so that $M \in (G_i,D_i)$ and $B_M < B_{G_i} = B_{D_i}$. By definition of $G_i$, we have $M-G_i>1$. This implies that $M \in \mathcal{RE}$, which contradicts the definition of $D_i$. Hence, $B_t \geq B_{G_i} = B_{D_i}$ for all $t \in [G_i,D_i]$, or equivalently $(B_t; G_i \leq t \leq D_i)$ is composed of excursions above the level $B_{G_i} = B_{D_i}$. 

If there exists an excursion interval $[u,v] \subset [G_i,D_i]$ with $v-u >1$, then by path continuity, $[u,v]$ contains at least a left end and a right end. This leads to a contradiction. Further, if there exist $G_i<u<v<D_i$ such that $B_{G_i} = B_u =B_v =B_{D_i}$, then $u$ and $v$ are two local minima at the same level. This is impossible, since a.s. the levels of  local minima in Brownian motion are all different, see Kallenberg \cite[Lemma 13.15]{Kallenberg}. Thus, $(B_u-B_{G_i}; G_i \leq t \leq D_i)$ is composed of at most two excursions of lengths no larger than $1$.

Finally, observe that for all $t \in (G_i,G_i+1]$, $B_t > B_{G_i}$ and thereby $t \notin \mathcal{RE}$. This implies that $D_i - G_i \geq 1$. If $D_i - G_i = 1$, then there exists a reflected bridge of length $1$ in Brownian motion by a space-time shift. But this is excluded by Pitman and Tang \cite[Theorem 4]{PTpattern}.
\end{proof}

According to Proposition \ref{corodecomp}, we get the decomposition \eqref{DecompGDT}
such that 
\begin{itemize}
\item 
$[T_{i}, T_{i+1}] \cap \mathcal{LE} = [T_i, G_i] \cap \mathcal{LE}$, i.e. left ends of forward meanders of unit length are contained in $[T_i, G_i]$;
\item
$[T_{i}, T_{i+1}] \cap \mathcal{RE} = [D_i, T_{i+1}] \cap \mathcal{RE}$, i.e. right ends of backward meanders of unit length are contained in $[D_i, T_{i+1}]$;
\item
$(G_i, D_i) \cap \mathcal{LE} \cap \mathcal{RE} = \emptyset$, i.e. $(G_i, D_i)$ contains neither left ends of forward meanders nor right ends of backward meanders.
\end{itemize}

\begin{proof}[Proof of Theorem \ref{ABC}]
Observe that $T_i$ is the $i^{th}$ time that the argmin process $(\alpha_t; t \geq 0)$ reaches $0$ by a continuous passage from $1$.
It is obvious that $D_i$ is a stopping time relative to $(\mathcal{G}_t)_{t \geq 0}$, the filtrations of the argmin process $\alpha$.  So $T_{i+1} - D_i$ is independent of $(G_i - T_{i}, D_i - G_i)$. Further by time reversal of $\alpha$ (Proposition \ref{timereversal}), we see that $G_i - T_{i}$, $D_i - G_i$ and $T_{i+1} - D_i$ are mutually independent, and $G_i - T_{i} \stackrel{(d)}{=} T_{i+1} - D_i$.

By Lemma \ref{lct}, $(1 - \alpha_{D_i+t}; 0 \leq t \leq T_{i+1} - D_i)$ has the same distribution as the age process of excursions above the past-minimum of Brownian motion until the age reaches $1$. As seen in Section \ref{s22}, $T_{i+1} - D_i \stackrel{(d)}{=} J$.
So $\mathbb{E}(T_{i+1} - D_i) = 1$ by \eqref{meanfirst}.

Recall the definitions of $\mu^{\uparrow 1}(x)$, $\Pi^{0 \uparrow}(dx)$ and $s(x,y)$ from \eqref{jrt1}, \eqref{jrt0} and \eqref{interh}.
From the L\'evy system of the argmin process $(\alpha_t; t \geq 0)$, we have
\begin{align*}
\mathbb{P}(D_i - G_i -1 \in dt) / dt & = c \int_t^1 \Pi^{0 \uparrow}(dx) s(x,x-t) \mu^{\uparrow 1}(x-t) dx \\
                                  & = c \sqrt{\frac{2}{\pi}} \frac{1-t}{(1+t)^2 \sqrt{t}},
\end{align*}
with some constant $c>0$. Further, $\int_0^1 \mathbb{P}(D_i - G_i -1 \in dt) = 1$ leads to $c = \sqrt{\pi /2}$. From this follows \eqref{lawDG}. Thus,
$$\mathbb{E}(D_i -G_i) = \int_1^2 t \cdot \frac{2-t}{t^2 \sqrt{t-1}} dt = \pi-2.$$
Alternatively,
$\mathbb{E}(D_i -G_i) = \mathbb{E}(T_{i+1}-T_i) - \mathbb{E}(G_i -T_i) - \mathbb{E}(T_{i+1} - D_i) = \pi -2$ by \eqref{Gapmean}, and the fact $\mathbb{E}(G_i - T_i) = \mathbb{E}(T_{i+1} -D_i) = 1$. 
\end{proof}
\begin{remark}
\em{ The path decomposition of Theorem \ref{ABC} provides an alternative way to compute the Laplace transform of the inter-arrival time $\Delta$. In fact,
\begin{equation}
\label{pdagree}
\Phi_{\Delta}(\lambda) = e^{-\lambda} (\Phi_{J}(\lambda))^2 \int_0^1 \frac{e^{-\lambda t}(1-t)}{(t+1)^2 \sqrt{t}} dt.
\end{equation}
The equality between \eqref{pdagree} and \eqref{Laplace3} reduces to the following identity
\begin{equation}
\label{pdagree2}
\int_0^1 \frac{e^{-\lambda t}(1-t)}{\sqrt{t}(1+t)^2} dt = \pi \lambda e^{\lambda} \left[(\erf(\sqrt{\lambda})^2 -1)\right]+ 2 \sqrt{\pi \lambda} \erf(\sqrt{\lambda})   + e^{-\lambda},
\end{equation}
which can also be verified analytically.
}
\end{remark}

Recall from Section \ref{s34} that $(l^1_t; t \geq 0)$ is the local times of $\alpha$ at level $1$, and $(l^0_t; t \geq 0)$ is the local times of $\alpha$ at level $0$. As a consequence of Theorem \ref{ABC}, we have
\begin{corollary}
\label{lctrate}
\begin{equation}
\label{lctrateeq}
\mathbb{E}l^1_t/t = \mathbb{E} l^0_t/t = \frac{1}{\sqrt{2 \pi}} \quad \mbox{for all}~ t>0.
\end{equation}
\end{corollary}
\begin{proof}
By stationarity of the argmin process $(\alpha_t; t \geq 0)$,
\begin{equation*}
\mathbb{E}l^1_t / t = \mathbb{E}(l^1_{[T_i, T_{i+1}]}) / \mathbb{E}(T_{i+1} - T_i),
\end{equation*}
where $l^1_{[T_i, T_{i+1}]}$ is the local times of $\alpha$ at level $1$ between $[T_i, T_{i+1}]$. Note that $l^1_{[T_i, T_{i+1}]} = l^1_{[D_i, T_{i+1}]}$. By Lemma \ref{lct} and L\'evy's theorem, $l^1_{[D_i, T_{i+1}]}$ has the same distribution as the first level above which occurs an excursion of length exceeding $1$. As seen in Section \ref{s22}, the latter is exponentially distributed with rate $\sqrt{2 / \pi}$. Thus, $\mathbb{E}(l^1_{[T_i, T_{i+1}]}) = \sqrt{\pi /2}$. Moreover, $\mathbb{E}(T_{i+1} - T_{i}) = \pi$ by \eqref{Gapmean}. From these follows \eqref{lctrateeq}.
\end{proof}
\section{The argmin process of random walks and L\'evy processes}
\label{s5}
\subsection{The argmin process of random walks}
In this part, we prove Theorem \ref{discreteth}. Recall the definition of the argmin chain $(A_N(n); n \geq 0)$ from \eqref{argminchain}. Fix $N  \geq 1$. Let $(\overset{\rightarrow}{X}_N(n); n \geq 0)$ be the moving window process of length $N$, defined by
$$\overset{\rightarrow}{X}_N(n):=(X_{n+1}, \ldots, X_{n+N}) \quad \mbox{for}~n \geq 0,$$
with associated partial sums $\overset{\rightarrow}{S^X_N}(n):=(0, X_{n+1}, X_{n+1}+X_{n+2}, \ldots, \sum_{i=1}^N X_{n+i})$. Similarly, let $(\overset{\leftarrow}{X}_N(n); n \geq 0)$ be the reversed moving window process of length $N$, defined by
$$\overset{\leftarrow}{X}_N(n):=(-X_{n}, \ldots, -X_{n-N+1}) \quad \mbox{for}~n \geq N,$$
with associated partial sums $\overset{\leftarrow}{S^X_N}(n):=(0, -X_{n}, -X_{n}-X_{n-1}, \ldots, -\sum_{i=1}^N X_{n+1-i})$. 
Note that $n+A_N(n)$ is the last time at which the minimum of $(S_k; k \geq 0)$ on $[n+1,n+N]$ is attained. So $A_N(n)$ is a function of $\overset{\rightarrow}{S^X_N}(n)$ or $\overset{\leftarrow}{S^X_N}(n+N)$. 
The following path decomposition is due to Denisov.

\begin{theorem} [Denisov's decomposition for random walks, \cite{Denisov}] 
\label{DenisovRW}
Let $S_n:= \sum_{i=1}^n X_i$, where $X_i$ are independent random variables. For $N \geq 1$, let 
$$A_N:=\sup\left\{0 \leq i \leq N: S_i=\min_{1 \leq k \leq N} S_k\right\}$$
be the last time at which $(S_k; k \geq 0)$ attains its minimum on $[0,N]$. For each positive integer $a$ with $0 \le a \le N$, given the event $\{A_N = a\}$, 
the random walk is decomposed into two conditionally independent pieces:
\begin{enumerate}[(a).]
\item
$(S_{a-k}-S_{a}; 0 \leq k \leq a)$ has the same distribution as $\overset{\leftarrow \quad}{S^X_{a}}(a)$ conditioned to stay non-negative;
\item
$(S_{a+k}-S_{a}; 0 \leq k \leq N - a)$ has the same distribution as $\overset{\rightarrow \quad \quad}{S^X_{N - a}} (a)$ conditioned to stay positive.
\end{enumerate}
\end{theorem}

By Denisov's decomposition for random walks, it is easy to adapt the argument of Proposition \ref{Markov} to show that $(A_N(n); n \geq 0)$ is a time-homogeneous Markov chain on $\{0,1,\cdots, N\}$. 

Now we compute the invariant distribution $\Pi_N$, and the transition matrix $P_N$ of  the argmin chain $(A_N(n); n \geq 0)$ on $\{0,1, \ldots, N\}$. To proceed further, we need the following result regarding the law of ladder epochs, originally due to Sparre Andersen \cite{SA}, Spitzer \cite{Spitzer} and Baxter \cite{Baxter}. It can be read from Feller \cite[Chapter XII$.7$]{Fellervol2}.
\begin{theorem} \cite{SA, Fellervol2}
\label{GF}
\begin{enumerate}
\item
Let $\tau_n:=\mathbb{P}(S_1 \geq 0, \ldots, S_{n-1} \geq 0, S_n<0)$ and $\tau(s): = \sum_{n=0}^{\infty} \tau_n s^n$. Then for $|s|<1$,
$$\log \frac{1}{1-\tau(s)} = \sum_{n=1}^{\infty} \frac{s^n}{n} \mathbb{P}(S_n<0).$$
\item 
Let $p_n:= \mathbb{P}(S_1 \geq 0, \ldots, S_n \geq 0)$ and $p(s): = \sum_{n=0}^{\infty} p_n s^n$. Then for $|s|<1$,
$$p(s) = \exp\left(\sum_{n=1}^{\infty} \frac{s^n}{n} \mathbb{P}(S_n \geq 0)\right).$$
\item
Let $\widetilde{p}_n:= \mathbb{P}(S_1 > 0, \ldots, S_n > 0)$ and $\widetilde{p}(s): = \sum_{n=0}^{\infty} \widetilde{p}_n s^n$. Then for $|s|<1$,
$$\widetilde{p}(s) = \exp\left(\sum_{n=1}^{\infty} \frac{s^n}{n} \mathbb{P}(S_n > 0)\right).$$
\end{enumerate}
\end{theorem}

In the sequel, let $T_{-}:=\inf\{n \geq 1; S_n<0\}$ and $\widetilde{T}_{-}:=\inf\{n \geq 1; S_n \leq 0\}$ so that $p_n = \mathbb{P}(T_{-} > n)$ and $\widetilde{p}_n = \mathbb{P}(\widetilde{T}_{-} > n)$. 

\begin{proof}[Proof of Theorem \ref{discreteth}]
Observe that the distribution of the argmin of sums on $\{0,1,\cdots,N\}$ is the stationary distribution of the argmin chain. Following Feller \cite[Chapter XII.8]{Feller}, this is the discrete arcsine law
\begin{equation*}
\Pi_N(k) = p_k \widetilde{p}_{N-k} \quad \mbox{for}~0 \leq k \leq N.
\end{equation*}
Let $t_i:=\mathbb{P}(T_{-}=i) = p_{i-1} - p_i$ and $\widetilde{t}_i:=\mathbb{P}(\widetilde{T}_{-}=i) = \widetilde{p}_{i-1} - \widetilde{p}_i$ for $i >0$. Now we calculate the transition probabilities of the argmin chain. We distinguish two cases.
\vskip 6 pt
{\bf Case $1$.} The argmin chain starts at $0 < i \leq N$: $A_N(0) = i$. This implies that for all $k \in [1, i-1]$, $S_k \geq S_i$, and for all $k \in [i+1, N]$, $S_k > S_i$.
\begin{itemize}
\item 
If $S_{N+1}>S_i$, then the last time at which $(S_n)_{1 \leq n \leq N+1}$ attains its minimum is $i$, meaning that $A_{N}(1) = i-1$.
\item
If $S_{N+1} = S_i$, then the last time at which $(S_n)_{1 \leq n \leq N+1}$ attains its minimum is $N+1$, meaning that $A_N(1) = N$.
\end{itemize}
If we look forward from time $i$, $N+1$ is the first time at which the chain enters $(-\infty, 0]$. Consequently, for $0 < i \leq N$,
\begin{equation}
P_N(i,N) = \frac{\widetilde{t}_{N+1-i}}{\widetilde{p}_{N-i}} \quad \mbox{and} \quad P_N(i,i-1) = 1 - P_N(i,N),
\end{equation}
which leads to \eqref{trans1}.
\vskip 6 pt
{\bf Case $2$.} The argmin chain starts at $i=0$: $A_N(0) = 0$. For $0 \leq j <N$, let $j+1$ be the last time at which the minimum on $[1,N]$ is attained. 
\begin{itemize}
\item 
If $S_{N+1}>S_{j+1}$, then the last time at which $(S_n)_{1 \leq n \leq N+1}$ attains its minimum is $j+1$, meaning that $A_N(1)=j$.
\item
If $S_{N+1} = S_{j+1}$, then the last time at which $(S_n)_{1 \leq n \leq N+1}$ attains its minimum is $N+1$, meaning that $A_N(1)=N$.
\end{itemize}
If we look backward from time $j+1$, the origin is the first time at which the reversed walk enters $(-\infty, 0)$. So for $0 \leq j <N$,
\begin{equation}
P_{N}(0,j) = \frac{t_{j+1} \widetilde{p}_{N-j}}{\widetilde{p}_N},
\end{equation}
which yields \eqref{trans2}. The above formula fails for $j=N$, but $P_N(0, N) = 1 -\sum_{j=0}^{N-1}P_N(0,j)$. 
\end{proof}
\fbox{\bf $F$ is continuous and $\mathbb{P}(S_n>0) = \theta \in (0,1)$} From Theorem \ref{GF}, we deduce the well known facts that
\begin{equation*}
\log p(s) = \theta \sum_{n=1}^{\infty} \frac{s^n}{n} \Longrightarrow p(s) = (1-s)^{-\theta} = 1+ \sum_{n=1}^{\infty} \frac{(\theta)_{n \uparrow}}{n!} s^n,
\end{equation*}
where $(\theta)_{n \uparrow}: = \prod_{i=0}^{n-1} (\theta+i)$ is the {\em Pochhammer symbol}. This implies that
\begin{equation}
\label{pcts}
p_n = \widetilde{p}_n = \frac{(\theta)_{n \uparrow}}{n!} \quad \mbox{for all}~n>0.
\end{equation}
By injecting \eqref{pcts} into \eqref{stationaryPi}, \eqref{trans1} and \eqref{trans2}, we get \eqref{disarcsine}, \eqref{eq5354} and \eqref{eq55}. The formula \eqref{eq56} is obtained by the following lemma.

\begin{lemma}
\label{lem11}
\begin{equation*}
P_{N}(0,N) = \frac{2(1-\theta)}{N+1} - \frac{(1 - 2 \theta)(2 \theta)_{N \uparrow}}{(N+1)(\theta)_{N \uparrow}}.
\end{equation*}
\end{lemma}
\begin{proof}
Note that $P_N(0,N) = 1 - \sum_{j=0}^{N-1} P_N(0,j)$. Thus , it suffices to show that
\begin{equation*}
\sum_{j=0}^{N-1} p_j p_{N-j} - \sum_{j=0}^{N-1} p_{j+1} p_{N-j} = \frac{1}{(N+1) !} \Bigg[ (N-2 \theta-1)(\theta)_{N \uparrow} + (1-2 \theta) (2 \theta)_{N \uparrow}\Bigg].
\end{equation*}
Furthermore, for $|s|<1$,
\begin{equation*}
(1-s)^{-2\theta} = \left(\sum_{j=0}^{\infty}p_j s^j \right)^2 = \sum_{N=0}^{\infty} \left(\sum_{j=0}^N p_j p_{N-j} \right)s^j.
\end{equation*}
By identifying the coefficients on both sides, we get
\begin{equation*}
\sum_{j=0}^N p_j p_{N-j} = \frac{(2 \theta)_{N \uparrow}}{N !} \quad \mbox{and} \quad \sum_{j=0}^{N+1} p_jp_{N+1-j} = \frac{(2 \theta)_{N+1 \uparrow}}{(N+1) !},
\end{equation*}
which leads to the desired result.
\end{proof}

When $F$ is symmetric and continuous, the above results can be simplified. In this case, $\mathbb{P}(S_n \geq 0) = \mathbb{P}(S_n>0) = \frac{1}{2}$.
\begin{corollary}
Assume that $F$ is symmetric and continuous. Then the stationary distribution of the argmin chain $(A_N(n); n \geq 0)$ is given by
\begin{equation}
\Pi_N(k) = \binom{2k}{k} \binom{2N-2k}{N-k} 2^{-2N} \quad \mbox{for}~0 \leq k \leq N.
\end{equation}
In addition, the transition probabilities are
\begin{equation}
P_N(i,N) = \frac{1}{2(N+1-i)} \quad \mbox{and} \quad P_N(i,i-1) = \frac{2N +1 -2i}{2(N+1-i)} \quad \mbox{for}~0<i \leq N;
\end{equation}
\begin{equation}
P_N(0,j) = \frac{\binom{N}{j}^2}{2(j+1)\binom{2N}{2j}} \quad \mbox{for}~0 \leq j <N \quad \mbox{and} \quad P_N(0,N) = \frac{1}{N+1}.
\end{equation}
\end{corollary}
\fbox{\bf Simple symmetric random walks} In \cite[Chapter III.3]{Feller}, Feller found for a simple symmetric walk,
\begin{equation}
\label{ptil2}
\widetilde{p}_{2n} = \widetilde{p}_{2n+1} = \frac{(\frac{1}{2})_{n \uparrow}}{2 \cdot n!} \quad \mbox{for all}~n \geq 1,
\end{equation}
and 
\begin{equation}
\label{p2}
p_{2n-1} = p_{2n} = \frac{(\frac{1}{2})_{n \uparrow}}{ n!} \quad \mbox{for all}~n \geq 1.
\end{equation}

By injecting \eqref{ptil2} and \eqref{p2} into \eqref{stationaryPi}, \eqref{trans1} and \eqref{trans2}, we get \eqref{eq513}, \eqref{eq514} and \eqref{eq515}. The formula \eqref{eq516} is obtained by the following lemma.
\begin{lemma}
\label{lem12}
\begin{equation*}
P_N(0,N) = \left\{ \begin{array}{ccl}
         \frac{1}{N+1} & \mbox{if}~ N ~\mbox{is odd}; \\ [8 pt]
         \frac{2}{N+2} & \mbox{if} ~N ~\mbox{is even}.
                \end{array}\right.
\end{equation*}
\end{lemma}
\begin{proof}
Note that $P_N(0,N) = 1 -\sum_{j=0}^{N-1} P_N(0,j)$. Thus, it suffices to show that
\begin{equation*}
\sum_{j=0}^{N-1} p_j p_{N-j} - \sum_{j=0}^{N-1} p_{j+1}p_{N-j} = 
\left\{ \begin{array}{ccl}
         \frac{N}{N+1}\widetilde{p}_N & \mbox{if}~ N ~\mbox{is odd}; \\ 
         \frac{N}{N+2}\widetilde{p}_N & \mbox{if} ~N ~\mbox{is even}.
                \end{array}\right.
\end{equation*}
Furthermore, for $s<1$,
\begin{equation*}
\frac{1}{1-s} = \left(\sum_{j=0}^{\infty}p_j s^j \right) \left(\sum_{j=0}^{\infty}\widetilde{p}_j s^j \right) = \sum_{N=0}^{\infty} \left(\sum_{j=0}^N p_j \widetilde{p}_{N-j} \right)s^j.
\end{equation*}
By identifying the coefficients on both sides, we get
\begin{equation*}
\sum_{j=0}^N p_j \widetilde{p}_{N-j} = \sum_{j=0}^{N+1} p_j \widetilde{p}_{N+1-j} = 1,
\end{equation*}
which leads to the desired result.
\end{proof}
\subsection{The argmin process of L\'evy processes}
We consider the argmin process $(\alpha^X_t; t \geq 0)$ of a L\'evy process $(X_t; t \geq 0)$. 
According to the L\'evy-Khintchine formula, the characteristic exponent of $(X_t; t \geq 0)$ is given by
$$\Psi_X (\theta) : = ia \theta + \frac{\sigma^2}{2} \theta^2 + \int_{\mathbb{R}} (1-e^{i\theta x} + i \theta x 1_{\{|x|<1\}}) \Pi(dx),$$ 
where $a \in \mathbb{R}$, $\sigma \geq 0$, and  $\Pi(\cdot)$ is the L\'evy measure satisfying $\int_{\mathbb{R}} \min(1,x^2) \Pi(dx) < \infty$. The L\'evy process $X$ is a compound Poisson process if and only if 
$\sigma = 0$ and $\Pi(\mathbb{R}) < \infty$. 
In this case, the process $X$ has the following representation:
\begin{equation}
\label{CPP}
X_t = ct + \sum_{i=1}^{N_t} Y_i \quad \mbox{for all}~t>0,
\end{equation}
where 
$c = -a - \int_{|x|<1} x \Pi(dx)$,
$(N_t; t \geq 0)$ is a Poisson process with rate $\lambda$, and $(Y_i; i \geq 1)$ are independent and identically distributed random variables with cumulative distribution function $F$, independent of $N$ and satisfying $\lambda F(dx) = \Pi(dx)$.  
See Bertoin \cite{Bertoin} and Sato \cite{Sato} for further development on L\'evy processes.

Assume that $X$ is not a compound Poisson process with drift, which is equivalent to
\begin{itemize}
\item[(CD).]
For all $t>0$, $X_t$ has a continuous distribution; that is for all $x \in \mathbb{R}$,
$\mathbb{P}(X_t = x) = 0$.
\end{itemize}
See Sato \cite[Theorem 27.4]{Sato}.
For $A \in \mathcal{B}(\mathbb{R})$, let 
$T_A : = \inf\{t>0: X_t \in A\}$
be the hitting time of $A$ by $(X_t; t \geq 0)$. 
Recall that $0$ is regular for the set $A$ if $\mathbb{P}(T_A = 0) = 1$.
According to Blumenthal's zero-one law, $0$ is regular for at least one of the half-lines $(-\infty,0)$ and $(0,\infty)$. There are three subcases: 
\begin{itemize}
\item[(RB).]
$0$ is regular for both half-lines $(-\infty,0)$ and $(0,\infty)$;
\item[(R$+$).]
$0$ is regular for the positive half-line $(0,\infty)$ but not for the negative half-line $(-\infty,0)$;
\item[(R$-$).]
$0$ is regular for the negative half-line $(-\infty,0)$ but not for the positive half-line $(0,\infty)$.
\end{itemize}

Millar \cite{Millar78} proved that almost surely $(X_t; 0 \leq t \leq 1)$ achieves its minimum at a unique time $A \in [0,1]$, and
\begin{itemize}
\item
under the assumption (RB), $X_{A-} = X_A = \inf_{t \in [0,1]}X_t$ almost surely;  
\item
under the assumption (R$+$), $X_{A-} > X_A = \inf_{t \in [0,1]}X_t$ almost surely;
\item
under the assumption (R$-$), $X_A > X_{A-} = \inf_{t \in [0,1]}X_t$ almost surely.
\end{itemize}
The following result is a simple consequence of Millar \cite[Proposition 4.2]{Millar78}.

\begin{theorem} \cite{Millar78}
\label{M78}
Assume that $(X_t; 0 \leq t \leq 1)$ is not a compound Poisson process with drift. Let $A$ be the a.s. unique time such that
$
\inf_{t \in [0,1]} X_t = \min (X_{A-}, X_A).
$
Given $A$, the L\'evy path is decomposed into two conditionally independent pieces:
$$\left(X_{(A-t)-}-\inf_{u \in [0,1]}X_u; 0 \leq t \leq A\right) \quad \mbox{and} \quad \left(X_{A+t}-\inf_{u \in [0,1]}X_u; 0 \leq t \leq 1-A\right).$$
\end{theorem}

In \cite{Millar78}, Millar provided the law of the post-$A$ process $\left(X_{A+t}-\inf_{t \in [0,1]}X_t; 0 \leq t \leq 1-A\right)$ but he did not mention the law of the pre-$A$ process $\left(X_{(A-t)-}-\inf_{t \in [0,1]}X_t; 0 \leq t \leq A\right)$. Relying on Chaumont-Doney's construction \cite{CD10} of L\'evy meanders, Uribe Bravo \cite{UB14} proved that if $(X_t; 0 \leq t \leq 1)$ is not a compound Poisson process with drift and satisfies the assumption (RB), then
\begin{itemize}
\item
$\left(X_{(A-t)-}-\inf_{u \in [0,1]}X_u; 0 \leq t \leq A\right)$ is a L\'evy meander of length $A$;
\item
$(X_{A+t}-\inf_{u \in [0,1]}X_u; 0 \leq t \leq 1-A)$ is a L\'evy meander of length $1-A$.
\end{itemize}
This result generalizes Denisov's decomposition to L\'evy processes with continuous distribution. 
Since a compound Poisson process is a continuous random walk, with Denisov's decomposition for random walks, it is easy to extend Theorem \ref{M78} to:
\begin{corollary} 
\label{haha}
Let $(X_t; 0 \leq t \leq 1)$ be a real-valued L\'evy process. Let 
$$A : = \sup\left\{0 \leq s \leq 1: X_{s} = \inf_{u \in [0,1]} X_{u}\right\}$$
be the last time at which $X$ achieves its minimum on $[0,1]$. Given $A$, the path of $X$ is decomposed into two conditionally independent pieces:
$$\left(X_{(A-t)-}-\inf_{u \in [0,1]}X_u; 0 \leq t \leq A\right) \quad \mbox{and} \quad \left(X_{A+t}-\inf_{u \in [0,1]}X_u; 0 \leq t \leq 1-A\right).$$
\end{corollary}
With Corollary \ref{haha}, it is easy to adapt the argument of Proposition \ref{Markov} to prove that $(\alpha^X_t; t \geq 0)$ is a time-homogeneous Markov process. 

Now we turn to the stable L\'evy process. Let $(X_t; t \geq 0)$ be a stable L\'evy process with parameters $(\alpha,\beta)$, and neither $X$ nor $-X$ is a subordinator. It is well known that $0$ is regular for the reflected process $X- \underline{X}$. So It\^o's excursion theory can be applied to the process $X- \underline{X}$, see Sharpe \cite{Sharpe} for background on excursion theory of Markov processes. 

Let ${\bf n}(d\epsilon)$ be the It\^o measure of excursions of $X-\underline{X}$ away from $0$. Monrad and Silverstein \cite{MS79} computed the law of lifetime $\zeta$ of excursions under ${\bf n}$:
\begin{equation}
{\bf n}(\zeta > t) = c \frac{t^{\rho-1}}{\Gamma (\rho)}  \quad \mbox{and} \quad {\bf n}(\zeta \in dt) = c(1-\rho) \frac{t^{\rho-2}}{\Gamma(\rho)}
\end{equation}
for some contant $c>0$. Following Remark \ref{jremark}, we have:
\begin{proposition}
Let $(X_t; t \geq 0)$ be a stable L\'evy process with parameters $(\alpha,\beta)$, and neither $X$ nor $-X$ is a subordinator. Then the jump rate of the argmin process $\alpha^X$ per unit time from $x \in (0,1)$ to $1$ is given by
\begin{equation}
\mu^{\uparrow 1}(x) = \frac{1 - \rho}{1-x} \quad \mbox{for}~0 < x <1.
\end{equation}
\end{proposition}
Finally, by doing similar calculations as in Section \ref{s35}, we obtain the Feller transition semigroup \eqref{Qtrans} for $\alpha^X$. 


\begin{thebibliography}{99}

\bibitem{AbE}
Joshua Abramson and Steven~N. Evans.
\newblock Lipschitz minorants of {B}rownian motion and {L}\'evy processes.
\newblock {\em Probab. Theory Related Fields}, 158(3-4):809--857, 2014.

\bibitem{Baxter}
Glen Baxter.
\newblock Combinatorial methods in fluctuation theory.
\newblock {\em Z. Wahrscheinlichkeitstheorie und Verw. Gebiete}, 1:263--270,
  1962/1963.

\bibitem{BF}
Vladimir Belitsky and Pablo~A. Ferrari.
\newblock Ballistic annihilation and deterministic surface growth.
\newblock {\em J. Statist. Phys.}, 80(3-4):517--543, 1995.

\bibitem{BJ}
Albert Benveniste and Jean Jacod.
\newblock Syst\`emes de {L}\'evy des processus de {M}arkov.
\newblock {\em Invent. Math.}, 21:183--198, 1973.

\bibitem{Bertoin}
Jean {Bertoin}.
\newblock {\em {L\'evy processes.}}
\newblock Cambridge: Cambridge Univ. Press, 1996.

\bibitem{BP}
Jean Bertoin and Jim Pitman.
\newblock Path transformations connecting {B}rownian bridge, excursion and
  meander.
\newblock {\em Bull. Sci. Math.}, 118(2):147--166, 1994.

\bibitem{BY}
Philippe Biane and Marc Yor.
\newblock Quelques pr\'ecisions sur le m\'eandre brownien.
\newblock {\em Bull. Sci. Math. (2)}, 112(1):101--109, 1988.

\bibitem{Bolmeander}
Erwin Bolthausen.
\newblock On a functional central limit theorem for random walks conditioned to
  stay positive.
\newblock {\em Ann. Probability}, 4(3):480--485, 1976.

\bibitem{BRT}
P.~J. Brockwell, S.~I. Resnick, and R.~L. Tweedie.
\newblock Storage processes with general release rule and additive inputs.
\newblock {\em Adv. in Appl. Probab.}, 14(2):392--433, 1982.

\bibitem{CD10}
L.~Chaumont and R.~A. Doney.
\newblock Invariance principles for local times at the maximum of random walks
  and {L}\'evy processes.
\newblock {\em Ann. Probab.}, 38(4):1368--1389, 2010.

\bibitem{CP1}
E.~{\c{C}}inlar and M.~Pinsky.
\newblock A stochastic integral in storage theory.
\newblock {\em Z. Wahrscheinlichkeitstheorie und Verw. Gebiete}, 17:227--240,
  1971.

\bibitem{CP2}
E.~{\c{C}}inlar and M.~Pinsky.
\newblock On dams with additive inputs and a general release rule.
\newblock {\em J. Appl. Probability}, 9:422--429, 1972.

\bibitem{Cinlar}
Erhan {\c{C}}inlar.
\newblock A local time for a storage process.
\newblock {\em Ann. Probability}, 3(6):930--950, 1975.

\bibitem{DVJ}
D.~J. Daley and D.~Vere-Jones.
\newblock {\em An introduction to the theory of point processes. {V}ol. {I}}.
\newblock Probability and its Applications (New York). Springer-Verlag, New
  York, second edition, 2003.
\newblock Elementary theory and methods.

\bibitem{Davis}
M.~H.~A. Davis.
\newblock Piecewise-deterministic {M}arkov processes: a general class of
  nondiffusion stochastic models.
\newblock {\em J. Roy. Statist. Soc. Ser. B}, 46(3):353--388, 1984.
\newblock With discussion.

\bibitem{Denisov}
I.~V. Denisov.
\newblock A random walk and a wiener process near a maximum.
\newblock {\em Theory of Probability \& Its Applications}, 28(4):821--824,
  1984.

\bibitem{DIM}
Richard~T. Durrett, Donald~L. Iglehart, and Douglas~R. Miller.
\newblock Weak convergence to {B}rownian meander and {B}rownian excursion.
\newblock {\em Ann. Probability}, 5(1):117--129, 1977.

\bibitem{Dynkin}
E.~B. Dynkin.
\newblock {\em Markov processes. {V}ols. {I}, {II}}, volume 122 of {\em
  Translated with the authorization and assistance of the author by J. Fabius,
  V. Greenberg, A. Maitra, G. Majone. Die Grundlehren der Mathematischen
  Wissenschaften, B\"ande 121}.
\newblock Academic Press Inc., Publishers, New York; Springer-Verlag,
  Berlin-G\"ottingen-Heidelberg, 1965.

\bibitem{EvansPitman}
Steven~N. Evans and Jim Pitman.
\newblock Stationary {M}arkov processes related to stable
  {O}rnstein-{U}hlenbeck processes and the additive coalescent.
\newblock {\em Stochastic Process. Appl.}, 77(2):175--185, 1998.

\bibitem{Faggionato}
Alessandra Faggionato.
\newblock The alternating marked point process of {$h$}-slopes of drifted
  {B}rownian motion.
\newblock {\em Stochastic Process. Appl.}, 119(6):1765--1791, 2009.

\bibitem{Feller}
William Feller.
\newblock {\em An introduction to probability theory and its applications.
  {V}ol. {I}}.
\newblock Third edition. John Wiley \& Sons Inc., New York, 1968.

\bibitem{Fellervol2}
William Feller.
\newblock {\em An introduction to probability theory and its applications.
  {V}ol. {II}.}
\newblock Second edition. John Wiley \& Sons, Inc., New York-London-Sydney,
  1971.

\bibitem{GreenPit}
Priscilla Greenwood and Jim Pitman.
\newblock Fluctuation identities for {L}\'evy processes and splitting at the
  maximum.
\newblock {\em Adv. in Appl. Probab.}, 12(4):893--902, 1980.

\bibitem{Groeneboom}
Piet Groeneboom.
\newblock Brownian motion with a parabolic drift and {A}iry functions.
\newblock {\em Probab. Theory Related Fields}, 81(1):79--109, 1989.

\bibitem{HJ}
J.~Hoffmann-J{\o}rgensen.
\newblock Markov sets.
\newblock {\em Math. Scand.}, 24:145--166 (1970), 1969.

\bibitem{Itoex}
Kyosi It{\^o}.
\newblock Poisson point processes attached to markov processes.
\newblock In {\em Proc. 6th Berk. Symp. Math. Stat. Prob}, volume~3, pages
  225--240, 1971.

\bibitem{JacodS}
J.~Jacod and A.~V. Skorokhod.
\newblock Jumping {M}arkov processes.
\newblock {\em Ann. Inst. H. Poincar\'e Probab. Statist.}, 32(1):11--67, 1996.

\bibitem{Kallenberg}
Olav Kallenberg.
\newblock {\em Foundations of modern probability}.
\newblock Probability and its Applications (New York). Springer-Verlag, New
  York, second edition, 2002.

\bibitem{KS}
Ioannis Karatzas and Steven~E. Shreve.
\newblock {\em Brownian motion and stochastic calculus}, volume 113 of {\em
  Graduate Texts in Mathematics}.
\newblock Springer-Verlag, New York, second edition, 1991.

\bibitem{Kingman}
J.~F.~C. Kingman.
\newblock Homecomings of {M}arkov processes.
\newblock {\em Advances in Appl. Probability}, 5:66--102, 1973.

\bibitem{KJ}
N.~V. Krylov and A.~A. Ju{\v{s}}kevi{\v{c}}.
\newblock Markov random sets.
\newblock {\em Trudy Moskov. Mat. Ob\v s\v c.}, 13:114--135, 1965.

\bibitem{Leuridan}
Christophe Leuridan.
\newblock Un processus ponctuel associ\'e aux maxima locaux du mouvement
  brownien.
\newblock {\em Probab. Theory Related Fields}, 148(3-4):457--477, 2010.

\bibitem{Levybook}
Paul L{\'e}vy.
\newblock {\em Processus {S}tochastiques et {M}ouvement {B}rownien. {S}uivi
  d'une note de {M}. {L}o\`eve}.
\newblock Gauthier-Villars, Paris, 1948.

\bibitem{Maison}
Bernard Maisonneuve.
\newblock Exit systems.
\newblock {\em Ann. Probability}, 3(3):399--411, 1975.

\bibitem{Meyer71}
P.~A. Meyer.
\newblock Une mise au point sur les syst\`emes de {L}\'evy. {R}emarques sur
  l'expos\'e de {A}. {B}enveniste.
\newblock In {\em S\'eminaire de {P}robabilit\'es, {VII} ({U}niv. {S}trasbourg,
  ann\'ee universitaire 1971--1972)}, pages 25--32. Lecture Notes in Math.,
  Vol. 321. Springer, Berlin, 1973.

\bibitem{Millar78}
P.~W. Millar.
\newblock A path decomposition for {M}arkov processes.
\newblock {\em Ann. Probability}, 6(2):345--348, 1978.

\bibitem{MS79}
Ditlev Monrad and Martin~L. Silverstein.
\newblock Stable processes: sample function growth at a local minimum.
\newblock {\em Z. Wahrsch. Verw. Gebiete}, 49(2):177--210, 1979.

\bibitem{NeveuPitman}
J.~Neveu and J.~Pitman.
\newblock Renewal property of the extrema and tree property of the excursion of
  a one-dimensional {B}rownian motion.
\newblock In {\em S\'eminaire de {P}robabilit\'es, {XXIII}}, volume 1372 of
  {\em Lecture Notes in Math.}, pages 239--247. Springer, Berlin, 1989.

\bibitem{PitmanLevy}
J.~W. Pitman.
\newblock L\'evy systems and path decompositions.
\newblock In {\em Seminar on {S}tochastic {P}rocesses, 1981 ({E}vanston,
  {I}ll., 1981)}, volume~1 of {\em Progr. Prob. Statist.}, pages 79--110.
  Birkh\"auser, Boston, Mass., 1981.

\bibitem{Pitcal}
Jim Pitman.
\newblock The distribution of local times of a {B}rownian bridge.
\newblock In {\em S\'eminaire de {P}robabilit\'es, {XXXIII}}, volume 1709 of
  {\em Lecture Notes in Math.}, pages 388--394. Springer, Berlin, 1999.

\bibitem{PTpattern}
Jim Pitman and Wenpin Tang.
\newblock Patterns in random walks and {B}rownian motion.
\newblock In Catherine Donati-Martin, Antoine Lejay, and Alain Rouault,
  editors, {\em In Memoriam Marc Yor - S\'eminaire de Probabilit\'es XLVII},
  volume 2137 of {\em Lecture Notes in Mathematics}, pages 49--88. Springer
  International Publishing, 2015.

\bibitem{PYbis}
Jim Pitman and Marc Yor.
\newblock The two-parameter {P}oisson-{D}irichlet distribution derived from a
  stable subordinator.
\newblock {\em Ann. Probab.}, 25(2):855--900, 1997.

\bibitem{RY}
Daniel Revuz and Marc Yor.
\newblock {\em Continuous martingales and {B}rownian motion}, volume 293 of
  {\em Grundlehren der Mathematischen Wissenschaften}.
\newblock Springer-Verlag, Berlin, third edition, 1999.

\bibitem{RogersPitman}
L.~C.~G. Rogers and J.~W. Pitman.
\newblock Markov functions.
\newblock {\em Ann. Probab.}, 9(4):573--582, 1981.

\bibitem{Sato}
Ken-iti Sato.
\newblock {\em L\'evy processes and infinitely divisible distributions},
  volume~68 of {\em Cambridge Studies in Advanced Mathematics}.
\newblock Cambridge University Press, Cambridge, 1999.
\newblock Translated from the 1990 Japanese original, Revised by the author.

\bibitem{Sharpe}
Michael Sharpe.
\newblock {\em General theory of {M}arkov processes}, volume 133 of {\em Pure
  and Applied Mathematics}.
\newblock Academic Press, Inc., Boston, MA, 1988.

\bibitem{SA}
Erik Sparre~Andersen.
\newblock On sums of symmetrically dependent random variables.
\newblock {\em Skand. Aktuarietidskr.}, 36:123--138, 1953.

\bibitem{Spitzer}
Frank Spitzer.
\newblock A combinatorial lemma and its application to probability theory.
\newblock {\em Trans. Amer. Math. Soc.}, 82:323--339, 1956.

\bibitem{Th1995}
Hermann Thorisson.
\newblock On time- and cycle-stationarity.
\newblock {\em Stochastic Process. Appl.}, 55(2):183--209, 1995.

\bibitem{Tsirelson}
Boris Tsirelson.
\newblock Brownian local minima, random dense countable sets and random
  equivalence classes.
\newblock {\em Electron. J. Probab.}, 11:no. 7, 162--198 (electronic), 2006.

\bibitem{UB14}
Ger{\'o}nimo Uribe~Bravo.
\newblock Bridges of {L}\'evy processes conditioned to stay positive.
\newblock {\em Bernoulli}, 20(1):190--206, 2014.

\bibitem{Wata}
Shinzo Watanabe.
\newblock On discontinuous additive functionals and {L}\'evy measures of a
  {M}arkov process.
\newblock {\em Japan. J. Math.}, 34:53--70, 1964.

\bibitem{Zolo}
V.~M. Zolotarev.
\newblock {\em One-dimensional stable distributions}, volume~65 of {\em
  Translations of Mathematical Monographs}.
\newblock American Mathematical Society, Providence, RI, 1986.
\newblock Translated from the Russian by H. H. McFaden, Translation edited by
  Ben Silver.

\end{thebibliography}



\ACKNO{We thank two anonymous referees for their careful reading and valuable suggestions.}


\end{document}